\documentclass{amsart}

\usepackage{amsmath}
\usepackage{amssymb}
\usepackage{amsthm}
\usepackage{amscd}
\usepackage{enumerate}
\usepackage[pdfusetitle,unicode,hidelinks]{hyperref}
\usepackage{bbm}
\usepackage{etoolbox}
\usepackage{tikz-cd}
\usepackage{listings}

\title{On quadratically enriched excess and residual intersections}

\author{Tom Bachmann}
\address{T.~Bachmann, Department of Mathematics, LMU Munich, Germany}
\email{\href{mailto:tom.bachmann@zoho.com}{tom.bachmann@zoho.com}}

\author{Kirsten Wickelgren}
\address{K.~Wickelgren, Department of Mathematics, Duke University, Durham, NC, USA }
\email{\href{kirsten.wickelgren@duke.edu}{kirsten.wickelgren@duke.edu}}

\usepackage[textsize=scriptsize]{todonotes}
\usepackage{amsmath}
\usepackage{amssymb}
\usepackage{amsthm}
\usepackage{amscd}
\usepackage{bbm}
\usepackage{etoolbox}
\usepackage{tikz-cd}
\usepackage{listings}
\usepackage{comment}

\newtoggle{draft}

\usepackage[margin=1in]{geometry}
\newtheorem{proposition}{Proposition}
\newtheorem{corollary}[proposition]{Corollary}
\newtheorem{lemma}[proposition]{Lemma}
\newtheorem{theorem}[proposition]{Theorem}
\newtheorem{meta-theorem}[proposition]{Meta-Theorem}
\newtheorem{meta-corollary}[proposition]{Meta-Corollary}

\newtheorem*{conjecture*}{Conjecture}
\newtheorem*{theorem*}{Theorem}
\newtheorem*{corollary*}{Corollary}
\newtheorem*{proposition*}{Proposition}
\newtheorem*{lemma*}{Lemma}
\theoremstyle{definition}
\newtheorem{definition}[proposition]{Definition}

\newtheorem*{definition*}{Definition}
\newtheorem*{construction*}{Construction}
\theoremstyle{remark}
\newtheorem{remark}[proposition]{Remark}
\newtheorem{warning}[proposition]{Warning}
\newtheorem*{remark*}{Remark}
\newtheorem{question}[proposition]{Question}
\newtheorem{example}[proposition]{Example}

\newtheorem*{example*}{Example}

\numberwithin{proposition}{section}
\newcommand{\wequi}{\simeq}
\DeclareRobustCommand{\ul}{\underline}
\newcommand{\iHom}{\ul{\operatorname{Hom}}}
\newcommand{\RiHom}{\RR\ul{\operatorname{Hom}}}
\newcommand{\Hom}{\operatorname{Hom}}
\newcommand{\Ext}{\operatorname{Ext}}
\newcommand{\Spec}{\operatorname{Spec}}
\newcommand{\Proj}{\operatorname{Proj}}

\def\op{\mathrm{op}}

\newcommand{\RR}{\mathbb{R}}

\newcommand{\Z}{\mathbb{Z}}

\newcommand{\calO}{\mathcal{O}}
\newcommand{\calE}{\mathcal{E}}
\newcommand{\calL}{\mathcal{L}}
\newcommand{\calV}{\mathcal{V}}
\newcommand{\calU}{\mathcal{U}}
\newcommand{\calW}{\mathcal{W}}

\let\scr=\mathcal
\let\bb=\mathbb
\def\A{\bb A}
\def\P{\bb P}

\newcommand{\SH}{\mathcal{SH}}
\def\ph{\mathord-}

\newcommand\bbb[1]{\ensuremath{{\mathbb{#1}}}}

\newcommand{\GW}{\mathrm{GW}}
\newcommand{\Tr}{\mathrm{Tr}}
\newcommand{\W}{\mathrm{W}}

\newcommand{\Bl}{\mathrm{Bl}}

\newcommand{\image}{\mathrm{Im}}
\newcommand{\ind}{\mathrm{ind}}
\def\H{\mathrm{H}}

\newcommand{\lra}[1]{\langle #1 \rangle}

\newcommand{\colim}{\operatorname{colim}}

\def\CAlg{\mathrm{CAlg}}

\newcommand{\Rrank}{\operatorname{rank}}

\newcommand{\BH}{\operatorname{B}}

\newcommand{\Sym}{\operatorname{Sym}}

\newcommand{\coker}{\operatorname{Coker}}

\newcommand{\Jac}{\operatorname{Jac}}
\newcommand{\FF}{\mathbb{F}}

\newcommand{\codim}{\operatorname{codim}}
\newcommand{\depth}{\operatorname{depth}}
\def\Map{\mathrm{Map}}
\def\map{\mathrm{map}}
\newcommand{\Kos}{\mathrm{Kos}}
\newcommand{\schemeW}{Q}

\def\dual{*}
\def\Dual{\mathbb{D}}
\def\Dcat{\mathrm{D}}
\def\Dbc{\Dcat^{b,c}}
\def\Dperf{\Dcat^p}
\def\tmap{\tau}

\iftoggle{draft} {
\newcommand{\NB}[1]{\todo[color=gray!40]{#1}}
\newcommand{\tom}[1]{\todo[color=green]{#1}}

\newcommand{\kirsten}[1]{\todo[color=yellow]{#1}}
}{ 
\newcommand{\NB}[1]{}
\newcommand{\tom}[1]{}
\newcommand{\kirsten}[1]{}
\renewcommand{\todo}[1]{}
}

\DeclareFontFamily{T1}{cbgreek}{}
\DeclareFontShape{T1}{cbgreek}{m}{n}{<-6>  grmn0500 <6-7> grmn0600 <7-8> grmn0700 <8-9> grmn0800 <9-10> grmn0900 <10-12> grmn1000 <12-17> grmn1200 <17-> grmn1728}{}
\DeclareSymbolFont{quadratics}{T1}{cbgreek}{m}{n}
\DeclareMathSymbol{\qoppa}{\mathord}{quadratics}{19}
\DeclareMathSymbol{\Qoppa}{\mathord}{quadratics}{21}

\makeatletter
\newcommand{\iftoggleverb}[1]{%
  \ifcsdef{etb@tgl@#1}
    {\csname etb@tgl@#1\endcsname\iftrue\iffalse}
    {\etb@noglobal\etb@err@notoggle{#1}\iffalse}%
}
\makeatother

\begin{document}

\begin{abstract}
We use recent duality results of Eisenbud--Ulrich to give tools to study quadratically enriched residual intersections when there is no excess bundle. We use this to prove a formula for the Witt-valued Euler number of an almost complete intersection. We give example computations of quadratically enriched excess and residual intersections.

\end{abstract}

\maketitle

\tableofcontents

\section{Introduction}

\subsection{Excess and residual intersections in enumerative geometry}\label{subsectionInto:excess_residual_intersections}
Excess intersections occur when the zero locus of $s$ polynomial functions has smaller codimension than $s$. Removing a (large dimension) subscheme of the zero locus results in a residual intersection. A famous example is Chasles' Theorem that there are $3264$ smooth conics in the complex projective plane tangent to five generally chosen conics \cite{Chasles}. Without taking into account an excess intersection caused by double lines, B\'ezout's theorem suggests that there would be  $6^5 = 7776$ \cite[p.290]{EisenbudHarris}. 

It is remarkable that the number of such conics does not depend on the choice of the original five conics, provided these are chosen generally. This ``invariance of number" can be shown in two steps: First, the sum of the number of smooth conics tangent to the given five and a contribution from the double lines is the Euler number of a vector bundle which is independent of the choice of conics. Second, the contribution from the double lines is independent of the choice of conics. We will see both of these carry over in a more arithmetic context, enriched in bilinear or quadratic forms.

Results in classical enumerative geometry frequently assume the geometric configurations in question are defined over the complex numbers. Replacing the complex numbers with a non-algebraically closed field usually breaks invariance of number. $\A^1$-homotopy theory provides an enumerative geometry which can restore invariance of number over a nonalgebraically closed field $k$ in a meaningful way. In this context, counts of geometric configurations are frequently valued in the Grothendieck--Witt group $\GW(k)$ of $k$, defined to be the group completion of isomorphism classes of symmetric, nondegenerate, $k$-valued bilinear forms \cite{kass2016class} \cite{Levine-EC}. 

While excess intersection formulas are available in this context \cite{Fasel-excess} \cite{DJK} \cite[Remark 5.22]{BW-A1Eulerclasses}, few computations in $\A^1$-enumerative geometry have been carried out, despite their frequency in classical enumerative geometry over $\mathbb{C}$. A notable exception is Sabrina Pauli's computation of the contribution to the enriched count of lines on a quintic $3$-fold from infinite families of lines on the Fermat quintic $3$-fold \cite{Pauli-Quintic_3_fold}. She uses an enriched dynamic intersection, and can likewise use this to compute a quadratically enriched count of the lines on the non-smooth cubic surface $\{ xyz=0 \}$ \cite[1.4.3.1]{Pauli-thesis}. 

In this paper, we develop tools for computing $\A^1$-excess and residual intersections from duality of coherent sheaves. $\mathbb{A}^1$-Euler numbers are related to coherent or Grothendieck--Serre duality by a Gauss-Bonnet formula \cite{levine2018motivic} valid for all vector bundles \cite{BW-A1Eulerclasses}. Their computation is thus closely related to computations of canonical modules and trace maps. A major inspiration for this paper is the recent work of Eisenbud and Ulrich on coherent duality in residual intersections \cite{eisenbud2019duality}. For a complete intersection $R=k[[x_1,\ldots,x_d]]/(f_1,\ldots, f_s)$, the canonical module is isomorphic to $R$ itself. For a separable Noether normalization $k[[x_1,\ldots, x_{d-s}]] \subseteq R$, the associated trace map corresponds to the Jacobian, giving a distinguished socle element of the $0$-dimensional ring $R/(x_1,\ldots, x_{d-s})$. Eisenbud and Ulrich relax the condition that $R$ is a complete intersection, allowing $R$ to contain an excess intersection corresponding to an ideal $I$ of codimension $g<s$ with $(f_1,\ldots, f_s) \subset I \subset k[[x_1,\ldots,x_d]]$ (satisfying a local condition on the number of generators and a depth condition). They investigate perfect pairings valued in the canonical module of the residual intersection $Z(\mathrm{Ann}(I/J))$ which itself has an explicit formula, giving duality results analogous to the case of the complete intersection. This situation should be mirrored in $\A^1$-enumerative geometry: for a regularly embedded zero locus, we have the excess intersection formula \cite{Fasel-excess}, building on work of Fulton and MacPherson \cite{Fulton-MacPherson78}. In the case of a $0$-dimensional complete intersection, the contribution to the Euler number corresponds to the Jacobian. Absent the regular embedding condition, we turn to other methods.  

\subsection{Organization}
The remaining subsections of this introduction summarize our main results.
Each subsection corresponds to one section of the article.

\subsection{Conics in $\mathbb{P}^{n}$ vanishing on $\mathbb{P}^2$ (\S\ref{section:ConicsP2n+1vanishP2})}
In \S\ref{section:ConicsP2n+1vanishP2}, we give an example computation of the $\A^1$-Euler number of an excess bundle, giving a quadratically enriched answer to the question ``How many points lie in the intersection of $n$ general conics in $\bbb{P}^n_k$ containing a fixed plane, for $n$ an odd positive integer?" (We restrict to the case where $n$ is odd because this is an orientation condition required to obtain a meaningful quadratic enrichment.) The associated $\A^1$-Euler numbers are those of $\calO(2)^n$ on $\bbb{P}^n_k$ and the excess bundle on $\bbb{P}^2_k$. The former is a multiple of the hyperbolic form, but the latter is not.  It is the class of an explicit bilinear form given in Theorem \ref{thm:conics_containing_P2}. New to the $\A^1$-context is the failure of multiplicativity of the Euler class in short exact sequences to effectively compute the $\bbb{A}^1$-Euler number of the excess bundle. Instead, we use the long exact sequence in cohomology of coherent sheaves and results on the trace. In Appendix \ref{Section:M2code}, we provide Macaulay2 code to compute the resulting $\A^1$-excess intersection numbers. In contrast to the classical situation, the quadratically enriched excess intersection depends on the general choice of conics. See Corollary \ref{co:intersection_general_conics_containing_P2_non_invariance_number}. Given classical excess intersection enumerative results, even with orientation conditions satisfied, the question is not ``What is the quadratic enrichment?" but ``Is there a quadratic enrichment?"

\subsection{An arithmetic enrichment of the $5$-conics problem (\S\ref{section:conics_tgt_5_conics})}\label{subsection:Intro5conics}
In \S\ref{section:conics_tgt_5_conics} we show the answer to this question is ``yes" for Chasles's theorem on the $5$-conics problem. For a ring $R$, let $\GW(R)$ denote the Grothendieck--Witt group of $R$, defined to be the group completion of the semi-ring under $\oplus$ and $\otimes$ of isomorphism classes of symmetric, nondegenerate bilinear forms $V \times V \to R$, where $V$ is a locally free $R$-module. For a unit $r$ in $R^\times$, let $\lra{r}$ in $\GW(R)$ denote the class of the bilinear form $R \times R \to R$ given by $(x,y) \mapsto rxy$. For a finite, \'etale extension $R \to S$, let 
\begin{equation}\label{eq:defn:tr}
\Tr_{S/R}: \GW(S) \to \GW(R) 
\end{equation} denote the {\em transfer} map which takes the class of a bilinear form $\beta: V \times V \to S$ to the composition $\tau_{S/R} \circ \beta: V \times V \to R$, where $\tau_{S/R}: S \to R$ denotes the trace map of the finite \'etale extension \cite[IV \S 5]{Lang-algebra}. Note that for all $r,s \in R^\times$\begin{equation}\label{lra_unchanged_squares}
\lra{r} = \lra{s^2r}
\end{equation}

The space of conics in $\mathbb{P}^2_k$ can be identified with $\mathbb{P}^5_k$. Given a conic $C'$ in $\mathbb{P}^2_k$, there is a degree $6$ equation determining the hypersurface of conics in $\mathbb{P}^5_k$ which are tangent to $C'$ \cite[\S8.1]{EisenbudHarris}. This data is a section of $\calO(6)$. For $5$ general conics, we obtain a section $\sigma$ of $\oplus_{i=1}^5 \calO(6)$. $\sigma$ has isolated zeros at the finite number of smooth conics $C$ tangent to all $5$ given conics. 

The {\em Jacobian determinant} of $\sigma$ at $C$, denoted $\Jac \sigma (C)$, is given by the determinant of the matrix of partials of the five degree $6$ equations in the variables given by local coordinates on $\bbb{P}^5$. (See Equation~\eqref{eq:Jac_ind_computation}.) It is naturally an element of the fiber at $C$ of the line bundle $\Hom(\det T\bbb{P}^5, \det \oplus_{i=1}^5 \calO(6)) $ on the space $\bbb{P}^5$ of conics, as discussed further in Section~\ref{subsection:notation_conventions}. Here, $T\bbb{P}^5$ denotes the tangent bundle. The canonical isomorphism  $$\Hom(\det T\bbb{P}^5, \det \oplus_{i=1}^5 \calO(6)) \cong  \calO (24) \cong \calO(12)^{\otimes 2}$$ identifies $\Jac \sigma (C)$ with an element of the fiber of $\calL^{\otimes 2}$ at $C$ where $\calL$ is the line bundle $\calO(12)$. (This is the data of a relative orientation of $\oplus_{i=1}^5 \calO(6)$. See Definition \ref{df:relative_orientation}.) Choosing any local trivialization of $\calL$, we may identify $\Jac \sigma (C)$ with a scalar in $k(C)$, and in fact the condition that the conics are general allows us to assume  $\Jac \sigma (C) \neq 0$. Because changing the trivialization of $\calL$ multiplies the scalar associated with  $\Jac \sigma (C)$ by a square, there is a well-defined element $\lra{\Jac \sigma (C)}$ in $\GW(k(C))$ by Equation~\eqref{lra_unchanged_squares}. In Theorem~\ref{thmconics_tgt_5_conics}, we prove the following: 

\begin{theorem}
Let $k$ be a perfect field of characteristic not $2$ and let $C_1, \ldots, C_5$ be general conics in $\bbb{P}_k^2$. There is equality in $\GW(k)$
\[
\sum_{C\text{ conic tangent to all }C_i} \Tr_{k(C)/k} \lra{\Jac \sigma (C)} = \frac{3264}{2}(\lra{1}+\lra{-1})
\]
\end{theorem}

Question~\ref{arithmetic-geometric-interpretation-local-degree-5-conics-problem} asks if there is an arithmetic-geometric interpretation of $ \lra{\Jac \sigma (C)}$ along the lines of \cite[Definition 9, Corollary 52]{CubicSurface}. We do not have an answer, but we have in mind the sort of interpretations in \cite{CubicSurface} \cite{FourLines} \cite{FinashinKharlamov-SegreIndices} \cite{Pauli-Quintic_3_fold}.

In this problem, we retain invariance of number in both of the ways mentioned in \S~\ref{subsectionInto:excess_residual_intersections}. First, there is the enriched $\bbb{A}^1$-Euler number $n(V)$ in $\GW(k)$ for a relatively oriented vector bundle $V$ mentioned previously; it is computed as the sum of a local index over zeros of an appropriate section, but is independent of the choice of section. This is discussed in more detail elsewhere, see for example \cite{Levine-EC} \cite{BW-A1Eulerclasses}. The second invariance, that of the excess contribution to $n(V)$, is not computed by the available excess intersection formulas in this case, because the double lines in the zero locus of $\sigma$ are not regularly embedded in the space of conics \cite[p. 463(d)]{EisenbudHarris} \cite[Chapter 6]{GriffithsHarris-Principles_algebraic_geometry}. Blowing up the locus of double lines as in  \cite[Chapter 6]{GriffithsHarris-Principles_algebraic_geometry} shows the contribution of the excess intersection does not depend on the choice of conics in this situation. See Remark \ref{rmk:Veronese_contribution_5_conics_problem}. It would be desirable to have more tools to address this sort of situation. We thus turn to considering coherent duality in the case of excess intersections following Eisenbud--Ulrich \cite{eisenbud2019duality}.

\subsection{Grothendieck duality and residual intersections (\S4)}
Let $R$ be a (commutative) ring and $J \subset I$ ideals.
We denote by $J:I = \mathrm{Ann}(I/J)$ the \emph{ideal quotient}.
More generally, if $X$ is a scheme and $Z \supset Z'$ are closed subschemes, we can consider a global version of this construction and define \[ Z : Z' = Z(I(Z) : I(Z')). \]
In some cases $Z:Z'$ has underlying set $\overline{Z \setminus Z'}$, and in general we think of $Z:Z'$ as an algebro-geometric enhancement of the closure of the complement of $Z'$ in $Z$.
This is a classical construction, closely related to primary decomposition (see e.g. \cite[Chapter 4]{Atiyah-Macdonald},\cite[Chapter 4 \S4 ]{CoxLittleOshea-IdealsVarietiesAlgorithms}).\NB{Also \S3.3.3 Ene and Herzog {\em Grobner basis in commutative algebra}}

The relationship to excess intersections is as follows.
Suppose that $Z$ is locally cut out by $s$ equations, but has codimension $<s$.
If $W=Z:Z'$ does have codimension $s$, then $W$ is called a residual intersection (of the pair $(Z,Z')$).
Assume from now on that $X$ is Gorenstein.
In favorable cases, $W \hookrightarrow X$ behaves somewhat like a regular immersion.
More concretely one hopes to find conditions on $Z'$ which ensure that whenever $W=Z:Z'$ is a residual intersection (for varying $Z$), then $W$ has certain good properties.
This program was initiated by Artin--Nagata \cite{ArtinNagata72}.
Assume for now that $X$ is local, $Z=Z(J), Z'=Z(I)$.
Ulrich \cite{ulrich1994artin} has found conditions (which are somewhat complicated and reviewed in \S\ref{subsub:hypotheses}) on $I$ which ensure that $W$ is \emph{Cohen--Macaulay}.
He moreover proves that then $I^{t+1}/JI^t$ is a canonical module for $W$, where $t = s-\codim I$.

We now reformulate this result using Grothendieck duality.
Write $i: W \to X$ for the closed immersion.
Since $X$ is Gorenstein, $i^!(\scr O_X)[s]$ is a canonical module for $W$, and hence by Ulrich's result there is a \emph{non-canonical} isomorphism $I^{t+1}/JI^t[-s] \wequi i^!(\scr O_X)$.
Eisenbud--Ulrich refine this in \cite[Theorem 4.6]{eisenbud2019duality} by writing down an explicit such isomorphism, which however still depends on choices (specifically a ``good'' set of generators for the ideal $J$).
We improve this by removing the choices, as follows.
\begin{theorem}[see Proposition \ref{pr:global_J/KJ_loc_free} and Theorem \ref{thm:global_canonical_trace}] \label{thm:intro-residual}
Let $W=Z:Z'$ be a residual intersection satisfying appropriate hypotheses.
Put $I=I(Z')$, $J=I(Z)$, $K=I(W)$ and $i:W \hookrightarrow X$.
\begin{enumerate}
\item The restricted conormal sheaf $C_{Z/X}|_W = J/J^2|_W=J/JK$ is a free $\scr O_W$-module of rank $s$.
\item There is a \emph{canonical} isomorphism \[ i^!(\scr O_X) \wequi I^{t+1}/JI^t[-s] \otimes \det(J/JK)^\dual. \]
\end{enumerate}
\end{theorem}
Note that if $Z$ has codimension $s$, so that $Z \hookrightarrow X$ is a regular immersion, we may choose $Z'=\emptyset$ and get $W=Z$.
In this case Theorem \ref{thm:intro-residual} reduces to well-known statements \cite[III \S1 and Corollary 7.3]{hartshorne1966residues}.
Note also that because our isomorphisms are canonical, it is no longer necessary to assume that $X$ is local.

One notable application of the above result is to the construction of pushforward maps.
If $p: Y \to X$ is a proper morphism of suitable schemes and $K$ is a dualizing complex on $X$, then one obtains a pushforward map in (e.g.) Grothendieck--Witt theory \begin{equation}\label{eq:pushforward} p_*: \GW(Y, p^! K) \to \GW(X, K), \end{equation} where by $\GW(X, L)$ we mean the hermitian $K$-theory group of $X$ with duality $\RiHom(\ph, L)$ \cite{CDHHLMNNS-1}.
If $X$ is Gorenstein then we may choose $K=\scr O_X$ and so the right hand side is just the Grothendieck--Witt group of symmetric bilinear bundles.
However $\GW(Y, p^!\scr O_X)$ remains difficult to access without a concrete identification of $p^! \scr O_X$.
Theorem \ref{thm:intro-residual} supplies this in some cases.

\subsection{Application to almost complete intersections (\S5)}
Let $X$ be a regular scheme of dimension $d$, $V$ a vector bundle on $X$ of rank $d$, and $\sigma$ a section of $V$.
Assume that $Z := Z(\sigma)$ has dimension $1$; in particular the section $\sigma$ is (mildly) \emph{degenerate}.
Thus the usual method (see e.g. \cite[Theorem 3]{CubicSurface}) for computing the Euler class $e(V,\sigma)$ no longer applies.
Let $Z'$ be the Cohen--Macaulay locus of $Z$ (obtained roughly speaking by removing embedded and isolated points).
Then $W=Z:Z'$ has dimension $0$ and hence is a residual intersection; thus the above theory applies (with $t=1$).
We have a symmetric bilinear form on $W$ valued in $I^2/JI$, given by $I/J$ and the multiplication map.\footnote{A priori this might be degenerate. Eisenbud--Ulrich prove non-degeneracy in \cite[Theorem 2.2]{eisenbud2019duality}.}
Using the pushforward from Equation \eqref{eq:pushforward} and the following paragraphs, we obtain $i_*(I/J)$, a form on $X$ valued in $\det V^\dual[d]$.
Reinterpreting work of van Straten--Warmt \cite{van2011gorenstein}, we prove the following.
\begin{theorem}[see Theorem \ref{thm:EU-vSW}]
If $W=Z:Z'$ satisfies the assumptions of Theorem \ref{thm:intro-residual}, then the form $i_*(I/J)$ is non-degenerate and its image in the \emph{Witt} group of $\det V^\dual[d]$-valued forms on $X$ coincides with the image of the Euler class $e(V, \sigma)$.
\end{theorem}
We stress that in the above result it is necessary to pass to Witt groups (instead of hermitian $K$-groups); see Example \ref{ex:O3+O2onP2_sectionsyx3}.

\subsection{Higher dimensions, duality with other assumptions, etc.} There are several paths forward from here (at least). M. Chardin, J. Na\'{e}liton and Q.H. Tran \cite{Chardin19} express the canonical modules of residual intersections as symmetric powers of $I/J$  under different hypotheses from \cite{eisenbud2019duality}, which may produce interesting pushforwards in Hermitian K-theory, and even perhaps formulas for $\A^1$-residual intersections. Along different lines, one could reasonably hope that Eisenbud and Ulrich's results on socle generators \cite[\S7]{eisenbud2019duality} give formulas for $\A^1$-Euler numbers, since this is the case for complete intersections. For some efforts in this direction, see \S\ref{subsec:explicit}, but we do not have examples of this. It would likewise be desirable to have an enriched solution to the $5$ conics problem following \cite{Fulton-MacPherson78} with a good enrichment of Segre classes of cones and contributions from the distinguished varieties of an intersection, but we have not attempted this.  Also notably missing from our results is a way to use Theorem \ref{thm:intro-residual} and \cite[Theorem 2.2]{eisenbud2019duality} in the case where $t>1$ to compute contributions to quadratically enriched counts, although we hope the globalization in Theorem \ref{thm:intro-residual} will be useful for this. What is missing is understanding what contribution might be given by $I$. When $Z(I)$ has dimension $1$ this contribution is hyperbolic, perhaps because the dimension is odd (cf. \cite[Proposition 3.4]{Levine-EC}, \cite[Proposition 19]{FourLines}).

\subsection{Notation, conventions, and some background}\label{subsection:notation_conventions}

\begin{definition}\label{df:relative_orientation}
A {\em relative orientation} of a vector bundle $V$ on a smooth scheme $X$ is the data of a line bundle $\mathcal{L}$ on $X$ and an isomorphism $\mathcal{L}^{\otimes 2} \cong \Hom(\det TX, \det V)$, where $TX$ denotes the tangent bundle of $X$. 
\end{definition}

For a relatively oriented $V$ on a smooth, proper $X$ over a field $k$, there is an $\bbb{A}^1$-Euler number $n^{\bbb{A}^1}(V)$ in $\GW(k)$. See for example \cite[Definition 3.7, Example 3.8]{BW-A1Eulerclasses}, \cite[Definition 35]{CubicSurface}, or \cite{Levine-EC}. (The ``n" in $n^{\bbb{A}^1}(\calE)$ stands for ``number" in contrast to the Euler ``class.") It is the sum of {\em local indices} $\ind_{Z} \sigma$ where $Z$ runs over the connected components of the zero locus $\{ \sigma = 0\}$ of a section $\sigma$ of $V$
\[
n^{\bbb{A}^1}(V) = \sum_{Z \subseteq \{ \sigma = 0\}} \ind_{Z} \sigma.
\]  As indicated by the notation, $n^{\bbb{A}^1}(V)$ is independent of the choice of section $\sigma$. We recall a few computations of $\ind_Z \sigma$. For further information and a full definition, see \cite[Definition 3.10]{BW-A1Eulerclasses} or \cite[Definition 30]{CubicSurface}.

When $Z=\{p\}$ is a zero of $\sigma$ which is \'etale over $k$, the index $\ind_{p} \sigma$ is computed as follows. The zero section determines a splitting of the short exact sequence of $k(p)$-vector spaces
\[
0 \to V(p) \to T_{\sigma(p)}V \to T_p X \to 0
\] where $V(p)$ denotes the fiber of $V$ at $p$ as a locally free sheaf. This splitting defines a projection $ T_{\sigma(p)}V \to V(p)$. Composing the map $T_p \sigma: T_p X \to T_{\sigma(p)}V$ with this projection, defines the {\em intrinsic derivative} of $\sigma$ at $p$
\begin{equation*}
\overline{T\sigma(p)}: T_{p} X \to V(p).
\end{equation*}  Thus $\Jac \sigma(p) :=\det \overline{T\sigma(p)}$ is an element of the fiber $\Hom(\det TX, \det V)(p) \cong \mathcal{L}(p)^{\otimes 2}$. A trivialization of $\mathcal{L}(p)$ identifies $\Jac \sigma(p) $ with an element of $k(p)$, which is non-zero by the assumption that $p$ is a simple zero of $\sigma$. Changing the trivialization of $\mathcal{L}(p)$ changes this element by a square. Thus we obtain a well-defined element $\lra{\Jac \sigma(p)} \in \GW(k(p))$. Then
\begin{equation}\label{eq:indpsigma_etale}
\ind_p \sigma = \Tr_{k(p)/k} \lra{\Jac \sigma(p)},
\end{equation} where $\Tr$ is as in Equation~\eqref{eq:defn:tr}.

$\Jac \sigma(p)$ is the Jacobian determinant of $\sigma$ viewed in local coordinates. See \cite[Section 4]{CubicSurface} for this point of view. The example of importance for this paper is that of \cite[3.1]{McKean-Bezout}, where $V = \oplus_{i =1}^n \calO_{\mathbb{P}^n}(d_i)$. Then $\sigma = (f_0,\ldots,f_n)$ where $f_i$ are homogeneous polynomials of degree $d_i$ in the standard homogenous coordinates $[x_0,\ldots,x_n]$. Suppose $p$ has $x_{\ell} \neq 0$. Define $f^{\ell}_i : = f_i((-1)^{\ell} \frac{x_0}{x_{\ell}}, \ldots, \frac{x_n}{x_{\ell}})$. Then 
\begin{equation}\label{eq:Jac_ind_computation}
\Jac \sigma(p) = \det \big(\frac{\partial f_i^{\ell}}{\partial(x_j/x_{\ell})} \big)_{i,j \neq \ell}.
\end{equation} See \cite[Lemma 5.5]{McKean-Bezout}.

Chow groups of a smooth scheme over a field $k$ admit a quadratic enrichment known as {\em Chow groups of oriented cycles} or {\em Chow--Witt} groups \cite[Section 2.1]{Fasel-LecturesChowWitt}. In analogy with the description of elements of Chow groups as rational equivalence classes of formal sums of subvarieties with integral coefficients, elements of Chow--Witt groups can be thought of as equivalence classes of formal sums of integral subvarieties with coefficients in the Grothendieck--Witt group of their function fields. These cohomology groups are represented by a $\bbb{P}^1$-spectrum $\H\tilde\Z$ called {\em generalized motivic cohomology} in stable $\bbb{A}^1$-homotopy theory $\SH(k)$ \cite{bachmann-very-effective}. The $\bbb{A}^1$-Euler number $n^{\bbb{A}^1}(V)$ just described is the pushforward of a refined Euler class $e(V,\sigma)$ in $\H\tilde\Z^{V^*}_{\{ \sigma = 0\}}(X):=[X/(X - \{\sigma = 0 \}),V/(V-0) \wedge \H\tilde\Z]_{\SH(k)}$ in Chow--Witt cohomology with support in $\{ \sigma = 0\}$ in degree $V^*$. See for example  \cite[\S5.3]{BW-A1Eulerclasses}.

There are (refined) Euler classes in other cohomology theories. For example, the Koszul complex of $(V,\sigma)$ with its standard duality determines a class $[e(V,\sigma)]$ in the $V$-twisted Witt group  (or $L$-group) $\W_Z^V(X)$ of $X$ with support in $Z$. The class $[e(V,\sigma)]$ is the refined Euler class in Witt theory. There is more information recalled on Witt theory in Section~\ref{subsection:Poincare_infinity_categories}.

Throughout \S\ref{sec:res-int} and \S\ref{sec:almost-complete-int}, we use small amounts of abstract homotopy theory.
In particular, we view derived categories as $\infty$-categories in the sense of \cite{lurie-htt}, that is, weakly enriched in spaces.
For example, given two maps of complexes which become equal after passing to a resolution, we call them \emph{homotopic}.
When talking about commutative diagrams, we always mean diagrams which are commutative up to homotopy.\footnote{And in fact we always mean that they are commutative up to a specific, but implicit, homotopy.}

Given a vector bundle $V$ of rank $n$, we denote by $\widetilde{\det} V$ the complex consisting of $\det V$ concentrated in degree $n$.

\subsection{Acknowledgements}
We wish to thank Claudia Polini for introducing us to the paper \cite{eisenbud2019duality}, David Eisenbud for useful discussions (he owes us money as he promised payment for asking questions he can answer), Stephen McKean for useful discussions on \cite{van2011gorenstein}, and Sabrina Pauli for useful discussions on quadratically enriched excess intersections. Kirsten Wickelgren was partially supported by NSF CAREER DMS 2001890 and NSF DMS 2103838. And she gratefully thanks Joseph Rabinoff for providing `` a room of one's own" to work on this paper.

\section{Conics in $\bbb{P}^{2n+1}$ vanishing on $\bbb{P}^2$}\label{section:ConicsP2n+1vanishP2}

Let $k$ be a field of characteristic not $2$.  We work in the category of schemes over $k$. Let $n$ be an odd positive integer and let $\Lambda \subset \bbb{P}^n$ be a reduced linear subscheme of dimension $2$, so $\Lambda$ is isomorphic to the projective plane $\Lambda \cong \bbb{P}^2$. Let $\{Q_1=0\},\ldots,\{Q_n=0\}$ be degree $2$ hypersurfaces containing $\Lambda$. By a {\em general choice} of such hypersurfaces, we mean a choice lying in some open subset. For such a general choice, the intersection $\cap_{i=1}^n \{ Q_i = 0\}$ as a scheme is the disjoint union of $\Lambda$ itself and a fixed number of isolated, reduced points. In other words, over the algebraic closure of $k$, we have $$\cap_{i=1}^n  \{ Q_i = 0\}= \Lambda \coprod \Gamma $$ where $\Gamma$ is a set of $2^n - {n \choose n-2} - {n \choose n-1} - {n \choose n}$ reduced points. See \cite[Proposition 13.5]{EisenbudHarris}.

We give an ``arithmetic count" or a ``quadratically enriched count" of the number of intersection points $\cap_{i=1}^n \{ Q_i = 0\}$ not lying in $\Lambda$ valued in $\GW(k)$. By this we mean that we will give each of the points of $\Gamma$ a weight in $\GW(k)$ associated to its field of definition and the quadrics $Q_i$ and compute the sum of these weights in $\GW(k)$. See also \cite{CubicSurface}, \cite{FourLines}, \cite{Levine-Witt}, \cite{McKean-Bezout}, \cite{pauli2020computing}, \cite{Pauli-Quintic_3_fold}, \cite{BW-A1Eulerclasses}, and \cite{enriched-inflection} for other arithmetic or quadratically enriched counts. We do this using a quadratically enriched excess intersection formula \cite[Remark 5.22(2)]{BW-A1Eulerclasses}, originally due to \cite{Fasel-excess} \cite{DJK}, and inspired by earlier work of Fulton and MacPherson among others \cite{fulton-intersection} \cite{Fulton-MacPherson78}. 

Excess intersection formulas rely on an excess bundle $\calE$ defined as follows. The $Q_i$ determine a section $\sigma=(Q_1,\ldots, Q_n)$ of $\oplus_{i=1}^n \calO(2)$ vanishing on $\Lambda \cup \Gamma$. Since the section $(Q_1,\ldots,Q_n)$ vanishes on $\Lambda$, there is a canonical map $N_{\Lambda} \bbb{P}^n  \to  \oplus_{i=1}^n \calO(2)$ from the normal bundle of $\Lambda$ in $\bbb{P}^n$ to the restriction of $ \oplus_{i=1}^n \calO(2)$ to $\Lambda$ given by forming $$\sigma^*:  \oplus_{i=1}^n \calO(2)^\dual \to I(\Lambda)$$ pulling back to $\Lambda$ and dualizing.  Explicitly, $N_{\Lambda} \bbb{P}^n \cong \calO(1)^{n-2}$, and without loss of generality, we may assume that $\Lambda$ is the simultaneous vanishing locus of $(x_3,\ldots,x_n)$ where $[x_0:\ldots:x_n]$ denote the standard homogeneous coordinates on $\bbb{P}^n$. Then $(x_3,\ldots,x_n)$ determines an isomorphism $N_{\Lambda} \bbb{P}^n \cong \oplus_{i=3}^n \calO(1)$ and we may write $Q_i = \sum_{j=3}^n \tilde{M}_{ij} x_j$, where the $\tilde{M}_{ij}$ are linear forms in the variables $x_0,\ldots, x_n$. The map $N_{\Lambda} \bbb{P}^n \cong \oplus_{i=3}^n \calO(1) \to \oplus_{i=1}^n \calO(2) $ is determined by the $n$ by $n-2$ matrix $M$ whose $(i,j)$th entry is $\tilde{M}_{ij}(x_0,x_1,x_2,0,0,\ldots,0)$, 
\begin{equation}\label{Mdef}
M=[\tilde{M}_{ij}(x_0,x_1,x_2,0,0,\ldots,0):\substack{i=1,\ldots,n \\ j= 3,\ldots,n}].
\end{equation} For general $(Q_1,\ldots, Q_n)$ vanishing on $\Lambda,$ this map is injective (because $\sigma^*$ is surjective or because $M$ is generically full rank) and thus extends to the short exact sequence \begin{equation}\label{SESexcessbundle}0 \to N_{\Lambda} \bbb{P}^n \cong \oplus_{i=3}^n \calO(1) \to \oplus_{i=1}^n \calO(2) \to \calE \to 0\end{equation} of vector bundles on $\Lambda$ defining the excess bundle $\calE$.

To apply a quadratically enriched excess intersection formula, the excess bundle needs a relative orientation in the sense of Definition~\ref{df:relative_orientation}. The bundle $\oplus_{i=1}^n \calO(2)$ has a canonical relative orientation, defined over $\Z$, induced from the canonical isomorphism $\calO(2) \cong \calO(1)^{\otimes 2}$. The excess bundle inherits a canonical relative orientation from the isomorphism $$\det \calE \otimes N_{\Lambda} \bbb{P}^n \cong \det \oplus_{i=1}^n \calO(2) $$ induced by the short exact sequence~\eqref{SESexcessbundle}, the relative orientation on $\oplus_{i=1}^n \calO(2)$, and the canonical isomorphism $\det T\Lambda \otimes \det  N_{\Lambda} \bbb{P}^n  \cong \det T\bbb{P}^n$. See for example \cite[Remark 5.22(2)]{BW-A1Eulerclasses}.

The Euler number (and the $\bbb{A}^1$-Euler number) of $\oplus_{i=1}^n \calO(2)$ is the sum of contributions from the points of $\Gamma$ and the Euler number of $\calE$. It follows that the number of isolated intersection points $\vert \Gamma \vert = 2^n - {n \choose n-2} - {n \choose n-1} - {n \choose n}$ is $2^n$ minus the classical Euler number of $\calE$, which can be computed by applying the multiplicativity of the total Chern class to \eqref{SESexcessbundle}. By contrast with the classical situation, this is not sufficient to compute the $\bbb{A}^1$-Euler number $n^{\bbb{A}^1}(\calE)$ in $\GW(k)$, for example because multiplicativity results are much less powerful, e.g. $\lra{a}(\lra{1}+\lra{-1}) = (\lra{1}+\lra{-1}) $. See Section~\ref{subsection:notation_conventions} for further information on the $\bbb{A}^1$-Euler number.

Instead we compute $n^{\bbb{A}^1}(\calE)$ using the Grothendieck--Serre description of the Euler number of $\calE$ of \cite[Definition 2.1]{BW-A1Eulerclasses}, which relies on the Koszul complex of $(\calE,0)$ with its canonical self-duality. (Here, $0$ denotes the zero section of $\calE$.) For the convenience of the reader, we give this description here in this special case. There is a canonical self-duality of the Koszul complex
\[
0 \to \wedge^2 \calE^* \to \calE^* \to \calO \to 0
\] valued in $\det \calE^* [2] = \wedge^2\calE^*[2]$, where $2$ is the rank of $\calE$, and the differentials in the Koszul complex are $0$. The relative orientation on $\calE$ provides the data of a line bundle $\calL$ together with an isomorphism 
\begin{equation}\label{eq:L2detEdual=omegaP2}
\calL^{\otimes 2} \otimes \det \calE^* \cong \omega_{\Lambda}  ,
\end{equation} where $\omega_{\Lambda} = \det \Omega_{\Lambda/k}$ is the canonical sheaf on $\Lambda$. So in our case $\omega_{\Lambda} \cong \calO(-3)$, $\det \calE^* \cong \calO(-n-2)$ and $\calL \cong \calO(m)$ with 
\[
m= \frac{n-1}{2}
\] (recall that $n$ is odd). The Koszul complex represents the Euler class of $\calE$ in Hermitian K-theory in an appropriate sense. The Euler number $n^{\bbb{A}^1}(\calE)$ is its pushforward and is represented by the following bilinear form.

We have a nondegenerate symmetric bilinear pairing on $H^0(\calO \otimes \calL) \oplus H^{2}(\wedge^{2} \calE^* \otimes \calL)$ in which the summands are isotropic and the pairing between the summands is given by
\[
H^0(\calO \otimes \calL) \times H^{2}(\wedge^{2} \calE^* \otimes \calL)  \stackrel{\cup}{\to} H^2(\wedge^2 \calE^* \otimes \calL^{\otimes 2}) \cong H^2(\omega_{\Lambda}) \stackrel{\Tr}{\to} k 
\] where $\Tr: H^2(\omega_{\Lambda}) \to k$ is the canonical trace map from Grothendieck--Serre duality. We have a further nondegenerate symmetric bilinear pairing 
\begin{equation}\label{middlenGS}
H^1(\calE^* \otimes \calL) \times H^1( \calE^* \otimes \calL) \stackrel{\cup}{\to} H^2(\wedge^2 \calE^* \otimes \calL^{\otimes 2}) \cong H^2(\omega_{\Lambda}) \stackrel{\Tr}{\to} k. 
\end{equation}
The sum of these forms over $k$ represents $n^{\bbb{A}^1}(\calE)$ in $\GW(k)$. Note that up to copies of the hyperbolic form, $n^{\bbb{A}^1}(\calE)$ is given by \eqref{middlenGS}.

We compute $n^{\bbb{A}^1}(\calE)$ using this description of the Euler number, combined with the long exact sequence in coherent cohomology associated to \eqref{SESexcessbundle} and formulas for the coherent duality trace. The computation of $n^{\bbb{A}^1}(\calE)$ is the main result in the quadratically enriched count of the points of $\Gamma$ in Theorem~\ref{thm:conics_containing_P2}. This computation will be carried out in the following steps. Step 1: Compute $H^1(\calE^* \otimes \calL)$ as a cokernel.  Step 2: Produce  \v{C}ech cocycle representatives of the elements of $H^1(\calE^* \otimes \calL)$. Step 3: Calculate the isomorphism $\wedge^2 (\mathcal{E}^* \otimes \mathcal{L}) \cong \omega_{\mathbb{P}^2}$ from the relative orientation of $\mathcal{E}$. Step 4: Calculate the traces of cup products of the cocycles from Step 2.

{\bf Step 1.} The first step in this process is to tensor the dual of \eqref{SESexcessbundle} with $\calL$ and compute $H^1(\calE^* \otimes \calL)$. Tensoring the dual of \eqref{SESexcessbundle} by $\calL$, we have the short exact sequence \begin{equation}\label{SEScalE*calL}
0 \leftarrow \oplus_{i=3}^n \calO(m-1) \stackrel{M^T}{\longleftarrow} \oplus_{i=1}^n\calO(m-2) \leftarrow \calE^* \otimes \calL \leftarrow  0 
\end{equation}

Let $k[x_0,\ldots,x_n]_d $ denote the $k$-vector space of homogeneous degree $d$ elements of $k[x_0,\ldots,x_n]$. The long exact sequence in cohomology of coherent sheaves associated to \eqref{SEScalE*calL} produces the isomorphism $$H^1(\calE^* \otimes \calL) \stackrel{\delta}{\leftarrow} \coker(k[x_0,x_1,x_2]^{n-2}_{m-1} \stackrel{M^T}{\longleftarrow} k[x_0,x_1,x_2]^n_{m-2} )$$ where $\delta$ is induced by the boundary map and $M^T$ is the transpose of the matrix $M$ of \eqref{Mdef}.

{\bf Step 2.} We will next produce an explicit description of a \v{C}ech cocycle representing $$\delta(h_1,\ldots,h_{n-2})$$ for any $(h_1,\ldots,h_{n-2})$ in $k[x_0,x_1,x_2]^{n-2}_{m-1}$. Define $N_0$,$N_1$,$N_2$ to be the principal minors of $M=[M_{ij}]$ given by removing the last two rows of $M$, the first and the last, and the first two respectively, i.e., $$N_0 = [M_{ij}:\substack{i=1,\ldots,n-2 \\ j= 3,\ldots,n}],~~~~~N_1 = [M_{ij}:\substack{i=2,\ldots,n-1 \\ j= 3,\ldots,n}],~~~~~N_2 = [M_{ij}:\substack{i=3,\ldots,n \\ j= 3,\ldots,n}].$$ For $i=0,1,2$, the minor $N_i$ gives rise to a splitting of \eqref{SEScalE*calL} over 
\begin{equation}\label{fdef}
U_i = \{f_i \neq 0\} \subset \bbb{P}^2,\quad \quad \text{where } f_i = \det N_i,
\end{equation} by the map \begin{equation}\label{sectionsNiT-1}\iota_i(N_i^T)^{-1}: \oplus_{j=3}^n \calO(m-1)\vert_{U_i} \to \oplus_{j=1}^n\calO(m-2)\vert_{U_i},\end{equation} where $$\iota_i: \oplus_{j=3}^n \calO(m-2)\vert_{U_i} \to \oplus_{j=1}^n\calO(m-2)\vert_{U_i}$$ is the inclusion into the summands corresponding to columns in $N_i$. Namely, $N_i^T$ is a $(n-2)\times(n-2)$ matrix of linear forms in the variables $x_0,x_1,x_2$. Its inverse can be expressed as a $(n-2)\times(n-2)$ matrix of rational functions of degree $-1$. (By Cramer's rule, each of these functions is the quotient of a polynomial of degree $n-3$ by $f_i$.) This inverse defines a map $$(N_i^T)^{-1}: \oplus_{j=3}^n \calO(m-1)\vert_{U_i} \to \oplus_{j=3}^n\calO(m-2)\vert_{U_i} .$$ 

The set of open subsets $\calU=\{U_0, U_1, U_2 \}$ is a cover of $\bbb{P}^2$ for general conics $Q_i$ because $\cap_{i=0}^2 (\bbb{P}^2-U_i)$ is the locus where $M$ is less than full rank, and must be empty because $N_{\Lambda}\bbb{P}^n \to \oplus_{j=1}^n \calO(2)$ is injective for general conics $Q_i$. (See also Equation~\eqref{Mdef}.) Let $$\delta^{\calU}: k[x_0,\ldots,x_n]^{n-2}_{m-1} \to C^1(\calU; \calE^* \otimes \calL)$$ denote the \v{C}ech boundary map associated to the short exact sequence \eqref{SEScalE*calL}, the cover $\calU$ and the splittings $N_i$ for $i=0,1,2$ \cite[III \S 4]{hartshorne2013algebraic}. We will let $U_{ij} = U_i \cap U_j$ and $U_{012} = U_0 \cap U_1 \cap U_2$ when describing cocycles. Additionally, we will view $\calE^* \otimes \calL$ as a subsheaf of $\calO(m-2)^n$ and write sections of $\calE^* \otimes \calL$ as the corresponding sections of $\calO(m-2)^n$. Let $p_i: \calO(m-2)^n \to \calO(m-2)$ denote the $i$th projection map. It turns out that for any $h = (h_1,\ldots,h_{n-2}) \in k[x_0,x_1,x_n]^{n-2}_{m-1}$, we will only need to compute $p_1(\delta^{\calU}h(U_{01}))$ and $p_n(\delta^{\calU}h(U_{12}))$ and we do this now. Using the sections \eqref{sectionsNiT-1}, we have \begin{align*}
\delta^{\calU}h(U_{01}) = \iota_1(N_1^T)^{-1} h - \iota_0 (N_0^T)^{-1} h \\
\delta^{\calU}h(U_{12}) = \iota_2(N_2^T)^{-1} h - \iota_1(N_1^T)^{-1} h 
\end{align*} Since $p_1\iota_1$ and $p_n\iota_1$ are both $0$, we have \begin{align*}
p_1\delta^{\calU}h(U_{01}) =  -p_1 \iota_0 (N_0^T)^{-1} h \\
p_n\delta^{\calU}h(U_{12}) = p_n\iota_2(N_2^T)^{-1} h  
\end{align*} The right hand side can be computed with Cramer's rule: \begin{align}
p_1\delta^{\calU}h(U_{01}) =  \frac{-1}{f_0} \det [h \vert (N_0^T)_1 ] \label{p1U01}\\
p_n\delta^{\calU}h(U_{12}) = \frac{1}{f_2}\det [ (N_2^T)_{n-2}\vert h  ]   \label{pnU12}
\end{align} where $(N_0^T)_1$ denotes the matrix $(N_0^T)$ with its first column removed and $(N_2^T)_{n-2}$ denotes the matrix $(N_2^T)$ with its last column removed. In other words, $(N_0^T)_1 = [M_{ij}:\substack{i=2,\ldots,n-2 \\ j= 3,\ldots,n}]^T$ denotes the $n-2$ by $n-3$ matrix which is the transpose of the matrix obtained from $M$ by removing the first and last two rows of $M$. Similarly, $(N_2^T)_{n-2} = [M_{ij}:\substack{i=3,\ldots,n-1 \\ j= 3,\ldots,n}]^T$ denotes the $n-2$ by $n-3$ matrix which is the transpose of the matrix obtained from $M$ by removing the first two and last rows of $M$. We have now obtained the needed information from \v{C}ech cocycle representatives of the image of $\delta$.

{\bf Step 3.} Recall the isomorphism~\eqref{eq:L2detEdual=omegaP2} $\calL^{\otimes 2} \otimes \det \calE^* \cong \omega_{\Lambda}$ produced by the canonical relative orientation on the excess bundle $\calE$ discussed previously.  We next compute the restriction of this isomorphism to $U_{012}$. (The motivation for this is to assist in the computation of $\Tr$ of elements in $H^2(\wedge^2(\calE^* \otimes \calL)) \cong H^2( \omega_{\Lambda})$ represented by \v{C}ech cocycles for the cover $\calU$.) The short exact sequence~\eqref{SEScalE*calL} determines a canonical isomorphism \begin{equation}
\label{detisoofcalE*calLSES}\wedge^2 (\calE^* \otimes \calL) \otimes \wedge^{ n - 2} \calO(m-1)^{n-2} \cong \wedge^n \calO(m-2)^n
\end{equation} \cite[\S 1.3 p. 10]{Fasel-LecturesChowWitt}. The isomorphism \eqref{detisoofcalE*calLSES} evaluated at $U_{012}$ can be described using any local splitting $s$ of \eqref{SEScalE*calL}. 

\[
\theta:\wedge^2 (\calE^* \otimes \calL)(U_{012}) \stackrel{\cong}{\longrightarrow} (\wedge^n \calO(m-2)^n) \otimes ( \wedge^{ n - 2} \calO(1-m)^{n-2} )(U_{012}) \cong \calO(-3)(U_{012})
\] induced from \eqref{detisoofcalE*calLSES} is 
\[
\theta(\alpha \wedge \beta) = (\alpha \wedge \beta) \wedge (\wedge^{n-2} s(e_1 \wedge \ldots \wedge e_{n-2})) \otimes (e_1 \wedge \ldots \wedge e_{n-2})^{-1}
\] where $\{e_1, \ldots, e_{n-2} \}$ is a basis of $\calO(m-1)^{n-2}(U_{012})$. (In particular, one can check that this is indeed independent of the choice of local basis and splitting.) Choose the splitting to be $s = \iota_1(N_1^T)^{-1}$. The map 
\[
\wedge^{n-2} s: \wedge^{n-2} \calO(m-1)^{n-2} \to \wedge^{n-2}  \calO(m-2)^{n} 
\] of locally free sheaves on $U_1$ factors through the direct summand $$\wedge^{n-2} \iota_1: \calO(m-2)^{\otimes (n-2)} \hookrightarrow  \wedge^{n-2}  \calO(m-2)^{n}$$ corresponding to the inclusion $\iota_1: \calO(m-2)^{(n-2)} \to \calO(m-2)^n $ of all but the first and last summands. Moreover, $\wedge^{n-2}(N_1^T)^{-1} = f_1^{-1}$ by definition.  Thus $(\wedge^{n-2} s)(\gamma) = (\wedge^{n-2} \iota_1)(\frac{\gamma}{f_1})$

Let $p_{1n}: \calO(m-2)^{n} \to \calO(m-2)^{2}$ denote projection onto the first and last factors. The second wedge power $$\wedge^2 p_{1n}: \wedge^2 \calO(m-2)^{n} \to \wedge^2  \calO(m-2)^{2}\cong  \calO(m-2) \otimes \calO(m-2) \cong \calO(n-5) $$ likewise defines a projection onto the direct summand. With this notation, we may write
\begin{equation}\label{inducedorientationiso}
\theta(\alpha \wedge \beta)= \frac{-\wedge^2 p_{1n}(\alpha \wedge \beta)}{f_1}
\end{equation} Recall that we are identifying $\alpha$ and $\beta$ as elements of $\calO(m-2)^{n}$ using the canonical inclusion $\calE^* \otimes \calL\hookrightarrow \calO(m-2)^{n}$. (Note that the image under $\theta$ only depends on $p_1$ of the first factor $\alpha$ and $p_n$ of the second factor $\beta$. This will turn into the justification of the previous assertion that we only need to compute $p_1(\delta^{\calU}h(U_{01}))$ and $p_n(\delta^{\calU}h(U_{12}))$ for $h = (h_1,\ldots,h_{n-2})$ in $k[x_0,x_1,x_n]^{n-2}_{m-1}$.)

{\bf Step 4.} Having obtained the desired isomorphism \eqref{inducedorientationiso} between $\det (\calE^* \otimes \calL)$ and $\calO(-3)$, we turn to computing the trace. The canonical trace map $\Tr: H^2(\Lambda, \calO(-3)) \to k$ pulls back to $$\Tr^{\calU}: C^2(\calU; \calO(-3)) \cong \calO(-3)(U_{012}) \to k$$ which is described as follows \cite[Chapters 6,7,8]{Kunz-ResiduesDualityAlgVar}. We have by assumption that $\calU$ is a cover, whence the ring $k[x_0,x_1,x_2]/(f_0,f_1,f_2)$ is a finite $k$-algebra, isomorphic to its localization at $(x_0,x_1,x_2)$. The degree $3(n-3)$ component $(k[x_0,x_1,x_2]/(f_0,f_1,f_2))_{3(n-3)}$ of $\frac{k[x_0,x_1,x_2]}{(f_0,f_1,f_2)}$ is a one-dimensional $k$-vector space, generated by the socle $E$, defined by 
\[
E = \det (a_{ij}) \quad \quad f_i = \sum_{j=0}^2 a_{ij} x_j \text{ for }i=0,1,2,
\] which satisfies the property that \[
\Jac (f_0,f_1,f_2) = \det\left(\frac{\partial f_i}{\partial x_j}\right)_{ij} = \dim_k (k[x_0,x_1,x_2]/(f_0,f_1,f_2)) \cdot E.\] The composition 
\begin{equation}\label{TrU-/f0f1f2}
k[x_0,x_1,x_2]_{3(n-3)} \stackrel{\frac{1}{f_0f_1f_2}}{\longrightarrow} \calO(-3)(U_{012}) \stackrel{\Tr^{\calU}}{\longrightarrow} k
\end{equation} sends $E$ to $1$, and factors through the quotient $$ k[x_0,x_1,x_2]_{3(n-3)}  \to (k[x_0,x_1,x_2]/(f_0,f_1,f_2))_{3(n-3)}.$$ In other words, we have the following \cite[Proposition 6.3, Definition 7.3, Theorem 7.4, p. 72, Proposition 8.32]{Kunz-ResiduesDualityAlgVar}.

\begin{proposition}\label{pr:TrUcomputation}
Suppose $f_0,f_1,f_2$ in $k[x_0,x_1,x_2]_{n-2}$ are such that $\calU = \{U_0, U_1, U_2 \}$ with $U_i = \{ f_i \neq 0 \}$ $i=0,1,2$ is a cover of $\bbb{P}^2$. The trace map with respect to the \v{C}ech-cover $\calU$
 $$ \Tr^{\calU}\left(\frac{(-)}{f_0f_1f_2}\right): k[x_0,x_1,x_2]_{3(n-3)} \to k $$ is the $k$-linear map determined by sending $(f_0, f_1,f_2)_{3(n-3)}$ to $0$ and $E$ to $1$. 
 \end{proposition}
 
 The Macaulay2 program CDTr given in \S\ref{Section:M2code} computes the function \eqref{TrU-/f0f1f2} of Proposition \ref{pr:TrUcomputation}. It is a modification of S. Pauli's program \cite{pauli2020computing} to compute the EKL class.

Combining the above descriptions of $\Tr^{\calU}$ and $\theta$ from Steps 4 and 3 respectively, we may explicitly compute the pairing $H^1(\Lambda, \calE^*) \times H^1(\Lambda, \calE^*) \to k$ equal to $n^{\bbb{A}^1}(\calE)$ in $W(k)$. Since $H^1(\Lambda, \calE^*)$ is a quotient of $k[x_0,x_1,x_2]^{n-2}_{m-1}$, this amounts to explicitly computing the pairing $$\BH': k[x_0,x_1,x_2]^{n-2}_{m-1} \times k[x_0,x_1,x_2]^{n-2}_{m-1} \to k $$ defined by $$(h,h') \mapsto \Tr( \delta(h) \cup \delta(h')).$$ Using our explicit \v{C}ech cocycles for $\delta(h)$ and $\delta(h')$, the cup product $ \delta(h) \cup \delta(h')$ is represented by a \v{C}ech cocycle consisting of the data of the element $$\delta^{\calU}h(U_{01})\vert_{U_{012}} \wedge \delta^{\calU}h'(U_{12})\vert_{U_{012}}  $$ of $(\calE^* \otimes \calL) \wedge (\calE^* \otimes \calL) (U_{012}) \hookrightarrow \calO(m-2)^n \wedge \calO(m-2)^n (U_{012})$. Combining equations \eqref{inducedorientationiso},  \eqref{p1U01}, and \eqref{pnU12}, we conclude that 
\begin{equation}\label{eq:Tr_cup}
 B'(h,h')=\Tr( \delta(h) \cup \delta(h')) = \Tr^{\calU} \left(\frac{\det [h \vert (N_0^T)_1] \det [ (N_2^T)_{n-2}\vert h'  ]}{f_0f_1f_2}  \right).
 \end{equation}

\begin{remark}
Note that each $f_i$ is of degree $n-2$ and  $\det [h \vert (N_0^T)_1]$ and $\det [ (N_2^T)_{n-2}\vert h'  ]$ are each of degree $n-3+m-1$. Thus $\frac{1}{f_0f_1f_2} \det [h \vert (N_0^T)_1] \det [ (N_2^T)_{n-2}\vert h'  ]$ is of degree $-3(n-2) + 2(n-3) + 2(m-1) = -n + n-1 -2 = -3$ as required.
\end{remark}

We express \eqref{eq:Tr_cup} in terms of $M$ as follows. Let $M^l_{i,j,v}$ denote the matrix $M$ with rows $i$, $j$, and $v$ removed along with column $l$. For $h$ in $k[x_0,\ldots,x_n]_{m-1}$ and $i=1,\ldots, n-2$, let $h_i$ denote the element of $k[x_0,x_1,x_2]^{n-2}_{m-1}$ with $h$ in the $i$th summand and $0$ elsewhere. The pairing $B'$ is then given $$(h_v,h_l') \mapsto  \Tr^{\calU} \left(\frac{(-1)^{v-1+n-2-1+l-1}h h'}{f_0f_1f_2} \det M^v_{1,n-1,n} \det M^{l}_{1,2,n} \right) $$ or more succinctly ($n$ is odd) \begin{equation}\label{answer}(h_v,h_l') \mapsto  \Tr^{\calU} \left(\frac{(-1)^{v+l}h h'}{f_0f_1f_2} \det M^v_{1,n-1,n} \det M^{l}_{1,2,n} \right)  \end{equation} for $h,h' \in k[x_0,\ldots,x_n]_{m-1}$. 

For $h$ in the image of $M^T: k[x_0,\ldots,x_n]^n_{m-2} \to k[x_0,\ldots,x_n]^{n-2}_{m-1}$, we have that $\delta h = 0$ in $H^1(\calE^* \otimes \calL)$. It follows that the pairing $\BH'$ computed by \eqref{answer} determines a pairing \begin{equation}\label{BilinearformonH1}\BH: \frac{k[x_0,x_1,x_2]^{n-2}_{m-1}}{M^T(k[x_0,x_1,x_2]^n_{m-2})} \times \frac{k[x_0,\ldots,x_n]^{n-2}_{m-1}}{M^T(k[x_0,\ldots,x_n]^n_{m-2})}  \to k \end{equation} by the same formula. We gather the above discussion into the following.

\begin{proposition}\label{A1Eulernumber_excess_bundle_n_quadrics_containingP2}
Let $k$ be a field of characteristic not $2$ and let $n$ be an odd positive integer. Let $Q_1,\ldots, Q_n$ in $k[x_0, x_1, \ldots, x_n]$ be homogeneous degree $2$ polynomials vanishing on $$\Lambda = \{ x_3 = x_4 = \ldots = x_n  =0\} \cong \bbb{P}^2 \hookrightarrow \bbb{P}^n \cong \Proj k[x_0, x_1, \ldots, x_n].$$ Define $M$ to be the $n$ by $n-2$ matrix of \eqref{Mdef}, i.e. $M$ describes the canonical map of vector bundles $N_{\Lambda} \bbb{P}^n \cong \oplus_{i=3}^n \calO(1) \stackrel{M}{\to} \oplus_{i=1}^n \calO(2)\vert_{\Lambda}$ associated to the closed immersion $\Lambda \hookrightarrow \{ Q_i = 0; i =1,\ldots,n\} $. Let $f_0$, $f_1$ and $f_2$ be the determinants of the principal minors of $M$ defined in \eqref{fdef}. Assume that $\cap_{i=0}^2 \{ f_i \neq 0\} = \emptyset$, which will occur for a general choice of  $Q_1,\ldots, Q_n$ vanishing on $\Lambda$. 

Then $\calE : = \coker M$ is a relatively oriented vector bundle on $\Lambda$, called the excess bundle. 

Define the bilinear form \[
\BH': k[x_0,x_1,x_2]^{n-2}_{m-1} \times k[x_0,x_1,x_2]^{n-2}_{m-1} \to k
\] by \eqref{answer} where $m=(n-1)/2$. Then $\BH'$ vanishes on $M^T(k[x_0,x_1,x_2]^n_{m-2})$ and descends to a non-degenerate, symmetric bilinear form $$[\BH]:\frac{k[x_0,x_1,x_2]^{n-2}_{m-1}}{M^T(k[x_0,x_1,x_2]^n_{m-2})} \times \frac{k[x_0,\ldots,x_n]^{n-2}_{m-1}}{M^T(k[x_0,\ldots,x_n]^n_{m-2})}  \to k,$$ and the $\bbb{A}^1$-Euler number of the excess bundle $\calE$ is given by the equality in $\GW(k)$\[
n^{\bbb{A}^1}(\calE) =[\BH] + \frac{1}{2}\left({n \choose 2} + n + 1 - \left( (n-2) {m-1+n\choose n} - n {m-2+n\choose n}\right)\right)(\lra{1}+ \lra{-1}).
\]  In particular, the image of $n^{\bbb{A}^1}(\calE)$ in $\W(k)$ is given by $[\BH]$.

\end{proposition}

\begin{proof}
We use the description of $ n^{\bbb{A}^1}(\calE)$ in \cite[Theorem 1.1 p. 3 Corollary, Definition 2.1]{BW-A1Eulerclasses} as recalled above. It follows that the image of $n^{\bbb{A}^1}(\calE)$ in $\W(k)$ is given by
\[H^1(\wedge^{i} \calE^* \otimes \calL) \times H^1(\wedge^{i} \calE^* \otimes \calL) \stackrel{\cup}{\to} H^r(\wedge^r \calE^* \otimes \calL^{\otimes 2}) \cong H^2(\omega_{\Lambda}) \stackrel{\Tr}{\to} k \] as likewise discussed above. We have just computed this pairing to be isomorphic to $\BH$. The value of $n^{\bbb{A}^1}(\calE)$ in $\GW(k)$ is the sum of $[\BH]$ and an integer multiple of $(\lra{1}+ \lra{-1})$. Explicitly
\[
n^{\bbb{A}^1}(\calE) = [\BH] + \frac{1}{2}\left(\Rrank(n^{\bbb{A}^1}(\calE)) - \Rrank(\BH)\right)(\lra{1}+ \lra{-1})
\] where $\Rrank: \GW(k) \to \Z$ denotes the rank function on the Grothendieck--Witt group. $\Rrank(n^{\bbb{A}^1}(\calE))$ is the Euler number in Chow of $\calE$ which is ${n \choose 2} + n + 1 $ by \cite[Proposition 13.5]{EisenbudHarris}. Then note that 
\begin{align*}
\Rrank (\BH) & = \dim  k[x_0,\ldots,x_n]^{n-2}_{m-1} - \dim k[x_0,\ldots,x_n]^n_{m-2}\\
& = (n-2) {m-1+n\choose n} - n {m-2+n\choose n}.
\end{align*}
\end{proof}

Proposition~\ref{A1Eulernumber_excess_bundle_n_quadrics_containingP2} leads to a quadratically enriched count of the intersection points of $n$ general quadrics containing a plane. Recall the transfer map on Grothendieck--Witt groups $\Tr_{E/L}: \GW(E) \to \GW(L)$ for a finite degree separable extension of fields $L \subseteq E$ of Equation~\eqref{eq:defn:tr} or more generally \cite[Definition 4.28]{A1-alg-top}. Suppose $Q_1, \ldots, Q_n \in k[x_0,\ldots,x_n]_2$, so $(Q_1, \ldots, Q_n)$ defines a section $\sigma$ of $\calO_{\bbb{P}^n}(2)^n$. For $p$ an \'etale zero of $(Q_1, \ldots, Q_n)$, let $ \Jac_p (Q_1, \ldots, Q_n)$ denote the Jacobian determinant of Section~\ref{subsection:notation_conventions} Equation~\ref{eq:Jac_ind_computation}.

\begin{theorem}\label{thm:conics_containing_P2}
Let $k$ be a perfect field of characteristic not $2$ and let $n$ be an odd positive integer. Let $Q_1,\ldots, Q_n$ in $k[x_0, x_1, \ldots, x_n]$ be homogeneous degree $2$ polynomials vanishing on $$\Lambda = \{ x_3 = x_4 = \ldots x_n  =0\} \cong \bbb{P}^2 \hookrightarrow \bbb{P}^n \cong \Proj k[x_0, x_1, \ldots, x_n].$$ For a general choice of such $Q_1,\ldots, Q_n$, we have $\cap_{i=1}^n  \{ Q_i = 0\}= \Lambda \coprod \Gamma $ where the scheme $\Gamma$ is \'etale over $k$ and 
\[
\sum_{p \in \Gamma} \Tr_{k(p)/k} \Jac_p (Q_1, \ldots, Q_n) = 2^{n-1}(\lra{1}+ \lra{-1}) - n^{\bbb{A}^1}(\calE).
\] where $n^{\bbb{A}^1}(\calE)$ is as in Proposition~\ref{A1Eulernumber_excess_bundle_n_quadrics_containingP2}.
\end{theorem}

This computation is entirely explicit and can be implemented on the computer for a fixed $n$. (To see this, note that in the definition \eqref{answer} of $B'$, the map $\Tr^{\calU}$ is computed explicitly by Proposition \ref{pr:TrUcomputation}.)

\begin{example}\label{n=5_quadrics_vanishingP2_explicit}
For $n=5$, we have quadrics $(Q_1,\ldots, Q_5)$ in $k[x_0,\ldots,x_5]^5$ vanishing on $\Lambda$. As above, define the matrix $M$ by $$M_{ij} = \tilde{M}_{ij}(x_0,x_1,x_2,0,\ldots,0)$$ where $Q_i = \sum_{j=3}^n \tilde{M}_{ij} x_j.$ In this example, $m=2$ and we give a more explicit formula for the pairing $$\BH': k[x_0,x_1,x_2]^{3}_{1} \times k[x_0,x_1,x_2]^{3}_{1} \to k $$  $$\BH' \cong \BH \oplus 5\lra{0}$$ of Proposition~\ref{A1Eulernumber_excess_bundle_n_quadrics_containingP2} and Theorem \ref{thm:conics_containing_P2}.  The matrix $M$ is of the form 
\[
M = \begin{bmatrix}
M_{13} & M_{14} & M_{15} \\
M_{23} & M_{24} & M_{25} \\
M_{33} & M_{34} & M_{35} \\
M_{43} & M_{44} & M_{45} \\
M_{53} & M_{54} & M_{55} 
\end{bmatrix}
\] Define $F_i= (-1)^i \det M^i_{1,n-1,n}$ and $G_i = (-1)^i \det M^{i}_{1,2,n}$ for $i=3,4,5$, so $F_i$ and $G_i$ are the following 
\[
F_3 =  -\begin{vmatrix}
M_{24} & M_{25} \\
 M_{34} & M_{35}
\end{vmatrix} \quad \quad F_4 =  \begin{vmatrix}
M_{23} & M_{25} \\
M_{33} & M_{35}
\end{vmatrix} \quad \quad F_5 = - \begin{vmatrix}
M_{23} & M_{24}  \\
M_{33} & M_{34}
\end{vmatrix} 
\]

\[
G_3 = - \begin{vmatrix}
M_{34} & M_{35}  \\
M_{44} & M_{45}
\end{vmatrix} \quad \quad G_4 =  \begin{vmatrix}
M_{33} & M_{35}\\
 M_{43} & M_{45}
\end{vmatrix} \quad \quad G_5 = - \begin{vmatrix}
M_{33} &  M_{34}\\
 M_{43}& M_{44}
\end{vmatrix} 
\]

Let $\{e_3,e_4, e_5\}$ denote the standard basis of  $k[x_0,x_1,x_2]^{3}_{1}$ as a $k[x_0,x_1,x_2]_{1}$-module. Then for any linear forms $x,x'$ in $k[x_0,x_1,x_2]_1$, and $i,j \in \{3,4,5\}$, we have 
\begin{equation}
\BH'(x e_i, x' e_j) = \Tr^{\calU}(\frac{x x'}{f_0 f_1 f_2} F_i G_j)
\end{equation}
Note that $\Tr^{\calU}$ is computed explicitly by Proposition~\ref{pr:TrUcomputation}.
\end{example}

$\BH'$ for $n=5$ can be computed with the Macaulay2 code of Appendix \ref{Section:M2code}. It is not difficult to find $n$-tuples of quadrics over a fixed field $k$ such that $\cap_{i=0}^2 \{ f_i \neq 0\} = \emptyset$. Using Proposition~\ref{A1Eulernumber_excess_bundle_n_quadrics_containingP2}, Theorem~\ref{thm:conics_containing_P2} and Macaulay2 it is likewise straightforward to compute the $\bbb{A}^1$-Euler numbers of the excess bundles. These $\bbb{A}^1$-Euler numbers depend on the choice of quadrics. 

\begin{example}\label{F31squaredisc}

Let $k = \bb F_{31}$. For quadrics $(Q_1,\ldots, Q_5)$ in $k[x_0,\ldots,x_5]^5$ vanishing on $\Lambda$, the $\bbb{A}^1$-Euler number of the excess bundle is determined by Example~\ref{n=5_quadrics_vanishingP2_explicit} or Proposition \ref{A1Eulernumber_excess_bundle_n_quadrics_containingP2}. To simplify notation let $x=x_0$, $y=x_1$ and $z=x_2$. Suppose $$M =  \begin{bmatrix} 
-15x-8y-14z & -2x-2y+10z & 12x-5y-10z \\
-8x-14y-2z & 10x-12y+13z &-15x-15y+13z\\
5x+11y+14z & -11x-10y-5z & 14x-2y+3z \\
7x-2y-4z  &  -10x+10z   & 4x-5y-4z\\
5x-10y+2z &  -12x+6y+7z & 8x-10y-11z\\
\end{bmatrix}.$$ This $M$ corresponds to $(Q_1,\ldots, Q_5)$ satisfying the hypothesis of Proposition~\ref{A1Eulernumber_excess_bundle_n_quadrics_containingP2}.

Then the bilinear form $\BH: k[x,y,z]^{3}_{1} \times k[x,y,z]^{3}_{1} \to k$  is given by the matrix $$ \begin{bmatrix}
-9 &15& 7 &  5 & -8 & 1 &  6 & -5 & 3 \\
15 & -1 &-9 & -8& -4 & 1 &  -5& 6 &  7\\
7 & -9 & -14 &1 & 1 &  6 &  3 & 7 &  -9\\
5 & -8 &1 &  -7& 0 &  1 &  2 & 13 & 2\\
-8 &-4 &1 &  0 &  -3 & 0 &  13 & -9 & -12\\
1 & 1 & 6 &  1 & 0 &  -3 & 2 & -12 &10\\
6 & -5 & 3 &  2 & 13 & 2 &13 & 9 &  4\\
 -5 &6 & 7 &  13 & -9 & -12 & 9 & 15 & 2  \\
3 & 7 & -9 & 2 & -12 & 10 & 4 & 2 &  9\\
\end{bmatrix}$$

Diagonalizing this matrix and removing the $5$ dimensional subspace on which it is trivial, we have that the $\bbb{A}^1$-Euler number of the excess bundle is $$\lra{-9}+\lra{-7}+\lra{-10}+\lra{-1} = \lra{-1}+\lra{-7}+\lra{-10}+\lra{-1}.$$
Macaulay2 checks that $70$ is a square mod $31$,
so the $\bbb{A}^1$-Euler number of the excess bundle is $$2(\lra{1}+\lra{-1}).$$

\end{example}

\begin{example}\label{F31nonsquaredisc}
Let $k = \bb F_{31}$. The quadrics $(Q_1,\ldots, Q_5)$ vanishing on $\Lambda$ corresponding to $$M =  \begin{bmatrix} 
3x-8y  &   -12x+9y+3z &  4x-12y+6z \\
2x-7y+8z  &   9x+2y-5z   &  14x-9y-4z\\
-11x-3y-2z  & 4x-5y+10z   & -13x+8y+8z  \\
-11x+15y+15z & -15x+10y-14z & 13x+10y+11z\\
-15x+14y-8z & 8x-6y+3z & 8x-3y-3z \\
\end{bmatrix},$$ satisfy the hypothesis of Proposition~\ref{A1Eulernumber_excess_bundle_n_quadrics_containingP2}. The bilinear form $\BH': k[x,y,z]^{3}_{1} \times k[x,y,z]^{3}_{1} \to k$ is computed with Macaulay2 to be $\BH' = \lra{-2} + \lra{3} + \lra{-14} + \lra{4} + 5\lra{0}$.  We have that the $\bbb{A}^1$-Euler number of the excess bundle is $$\lra{-2}+\lra{3}+\lra{-14}+\lra{1} $$ which has nonsquare discriminant. 
\end{example}

Examples \ref{F31squaredisc} and \ref{F31nonsquaredisc} imply:

\begin{corollary}\label{co:intersection_general_conics_containing_P2_non_invariance_number}
The quadratically enriched count of the points of $\Lambda$ given $$\sum_{p \in \Gamma} \Tr_{k(p)/k} \Jac_p (Q_1, \ldots, Q_n)$$ does not satisfy invariance of number, although its classical analogue $\sum_{p \in \Gamma} [k(p):p]$ does.
\end{corollary}

\begin{remark} 
Examples \ref{F31squaredisc} and \ref{F31nonsquaredisc} were found by constructing matrices $M$ with random (linear) entries over arbitrary human-chosen (smallish) finite fields, and verifying the required properties using Macaulay2.
\end{remark}

\section{Conics tangent to $5$ conics}\label{section:conics_tgt_5_conics}

Let $k$ be a field of characteristic not $2$, and let $V$ be a $3$-dimensional vector space over $k$. Consider the projective plane $\P V = \Proj \oplus_{i=0}^{\infty}\Sym^i V^* \cong\P^2_k $. The space of conics in $\P V$ can be identified with $\P \Sym^2 V^* \cong \P^5_k$. For a smooth conic $C$ in the plane $\P V$, the conics tangent to $C$ form a Cartier divisor $\sigma_C: \calO_{\P^5} \to \calO(6)$. See for example \cite[8.1(b)]{eisenbud2006geometry}. Let $Z_C \hookrightarrow \P^5$ denote the corresponding closed subscheme. Inside $Z_C$ are points corresponding to double lines, for which we use the following notation. 

Let $S \hookrightarrow \P^5_k$ denote the Veronese surface which is the scheme theoretic image of the Veronese map $\P V^* \to \P \Sym^2 V^*$ sending $\ell$ in $V^*$ to $\ell^2$ in $\Sym^2 V^*$. Let  \[\pi: X=\Bl_S \P^5 \to \P^5\] denote the blow-up of $\P^5$ along $S$, and let $E$ denote the exceptional divisor of $\pi$. For a generic conic, the proper transform $\tilde{Z}_C$ of $Z_C$ by the blow-up $\pi: X \to \P^5$ is a Cartier divisor with associated line bundle $\calO(-2E) \otimes \pi^* \calO(6)$. See for example \cite[p. 750-751]{GriffithsHarris-Principles_algebraic_geometry}. Let $\tilde{\sigma}_C: \calO_{X} \to \calO(-2E) \otimes \pi^* \calO(6)$ denote the corresponding section. Let $\calV \to X$ be the vector bundle on $X$ given by $\calV = \oplus_{i=1}^5 \calO(-2E) \otimes \pi^* \calO(6)$. 

There is an isomorphism $\omega_X \cong \omega_{\P^5} \otimes \calO((3-1)E)$ \cite[p. 188]{hartshorne2013algebraic} , whence an isomorphism $$\omega_X \cong \pi^* \calO(-6) \otimes \calO(2E) .$$ Since we have an isomorphism $$\det \calV \cong \otimes_{i=1}^5 (\calO(-2E) \otimes \pi^* \calO(6)),$$ we may combine to produce a relative orientation $$\omega_X \otimes \det \calV \cong \calO(-2E)^{\otimes 4} \otimes \pi^* \calO(6)^{\otimes 4} \cong L^{\otimes 2} $$ of $\calV/X$ where $L$ is the line bundle $L = \calO(-2E)^{\otimes 2} \otimes \pi^* \calO(6)^{\otimes 2}$. (See Definition \ref{df:relative_orientation} for the definition of a relative orientation.) It will not matter here which isomorphism we choose. Note this choice is equivalent to the choice of relative orientation. The set of choices of relative orientation here is a torsor under $k^*/(k^*)^2$. Since $\calV$ has an odd dimensional summand, we have that the $\bbb{A}^1$-Euler number $n^{\bbb{A}^1}(\calV)$ is a multiple of $\lra{1}+\lra{-1}$ \cite[Proposition 19]{FourLines} \cite[Proposition 3.4]{Levine-EC}. 

We have a canonical relative orientation of $\oplus_{i=1}^5 \calO(6)$ on $\P^5$ resulting from the canonical isomorphism $\omega_{\P^5/k} \otimes \det (\oplus_{i=1}^5 \calO(6)) \cong \calO(12)^{\otimes 2}$ as in \S\ref{subsection:Intro5conics}.

For a general choice of $5$ conics $C_1, \ldots, C_5$, the intersection $\cap_{i=1}^5 \tilde{Z}_{C_i}$ is disjoint from $E$ \cite[Assertion 3 p. 750, p.754]{GriffithsHarris-Principles_algebraic_geometry} or \cite[8.2.3, 8.2.1 p. 301]{EisenbudHarris} (To use the reference \cite[8.2.3]{EisenbudHarris}, note the space of complete conics is the blowup $\Bl_S \P^5$ by \cite[8.2.1 p. 301]{EisenbudHarris}, so our $X$ notation is consistent with theirs.) Moreover a point of $\cap_{i=1}^5 \tilde{Z}_{C_i}$ corresponds to a smooth conic $C$ \cite[8.2.3, p. 302-303]{EisenbudHarris}. By \cite[8.2.3 Lemma 8.7, 8.5]{EisenbudHarris} $\tilde{Z}_{C_i}$ is smooth at $C$ and the only point of mutual intersection of the tangent spaces $T_C\tilde{Z}_{C_i}$ for $i=1,\ldots,5$ is the origin $\cap_{i=1}^5 T_C\tilde{Z}_{C_i} = \{0\}$. It follows that $\cap_{i=1}^5 \tilde{Z}_{C_i}$ is \'etale at $C$. We thus have that for a general choice of $5$ conics $C_1, \ldots, C_5$, the intersection $\cap_{i=1}^5 \tilde{Z}_{C_i}$ is disjoint from $E$ and finite \'etale over $k$. The degree of $\cap_{i=1}^5 \tilde{Z}_{C_i}$ over $k$ is $3264$ by Chasles's theorem \cite[Theorem 8.9]{EisenbudHarris}. 

Define sections $\sigma$ of $\oplus_{i=1}^5 \calO(6)$ and $\tilde{\sigma}$ of $\calV$  by $\sigma = \oplus_{i=1}^5 \sigma_{C_i}$ and $\tilde{\sigma} = \oplus_{i=1}^5 \tilde{\sigma}_{C_i}$ respectively. Via the canonical trivialization of $\calO(E)$ away from $E$, we have $$\ind_C \sigma = \ind_C \tilde{\sigma} $$ for every $C$ in $\cap_{i=1}^5 \tilde{Z}_{C_i}$, where $\ind_C$ denotes the local index of the appropriate sections of the corresponding relatively oriented bundles at the isolated zero $C$. Note that $\ind_C \sigma$ can be computed by the formulas of Equations~\eqref {eq:indpsigma_etale} and~\eqref{eq:Jac_ind_computation}.

We therefore have $$\sum_{C\text{ conic tangent to all }C_i} \ind_C \sigma =  \sum_{C\text{ conic tangent to all }C_i} \ind_C \tilde{\sigma} = n^{\bbb{A}^1}(\calV).$$ Since $\cap_{i=1}^5 \tilde{Z}_{C_i}$ is finite, \'etale of degree $3264$ over $k$, the rank of the middle sum (say) is $3264$. Since $n^{\bbb{A}^1}(\calV)$ is a multiple of $\lra{1}+\lra{-1}$, it follows that $$n^{\bbb{A}^1}(\calV) = \frac{3264}{2}(\lra{1}+\lra{-1}).$$ This shows the following enriched version of the $5$-conics problem. 

\begin{theorem}\label{thmconics_tgt_5_conics}
Let $k$ be a perfect field of characteristic not $2$. For $C_1, \ldots, C_5$ general conics in $\bbb{P}_k^2$, there is equality in $\GW(k)$
\begin{equation*}
\sum_{C\text{ conic tangent to all }C_i}\Tr_{k(C)/k}\lra{ \Jac \sigma (C)} = \frac{3264}{2}(\lra{1}+\lra{-1}),
\end{equation*} where $\sigma = \oplus_{i=1}^5 \sigma_{C_i}$ is the section of $\oplus_{i=1}^5 \calO(6)$ corresponding to $C_1, \ldots, C_5$. 
\end{theorem}

\begin{question}\label{arithmetic-geometric-interpretation-local-degree-5-conics-problem}
Is there an interpretation of $$   \Tr_{k(C)/k}\lra{ \Jac \sigma (C)} = \ind_C \sigma$$ in terms of the arithmetic-geometry of $C,C_1,\ldots,C_5$?
\end{question}

\begin{remark}\label{rmk:Veronese_contribution_5_conics_problem}
In contrast to Corollary~\ref{co:intersection_general_conics_containing_P2_non_invariance_number}, this excess intersection problem satisfies invariance of number in both the classical and quadratically enriched settings. In more detail, the refined Euler class $e(\calW,\sigma) \in \H\tilde\Z^{\calW^*}_{\{ \sigma = 0\}}(\bbb{P}^5_k)$ of $\calW = \oplus_{i=1}^5\calO_{\bbb{P}^5}(6)$ in Chow--Witt cohomology with support in $\{ \sigma = 0\}$ has a component associated to the Veronese surface $S \subset \{ \sigma = 0\}$ via the isomorphism $\H\tilde\Z^{\calW^*}_{\{ \sigma = 0\}}(\bbb{P}^5_k) \cong \oplus_{Z_0 \subset \sigma = 0} \H\tilde\Z^{\calW^*}_{\{ Z_0\}}(\bbb{P}^5_k) $ where $Z_0$ runs over the connected components of $\{ \sigma = 0\}$. (See Section~\ref{subsection:notation_conventions} for some information on Chow--Witt groups and refined Euler classes.) Pushing forward these components expresses $n^{\bbb{A}^1}(\calW)$ as a sum in $\GW(k)$ of a contribution from the connected component of $\{ \sigma = 0\}$ with support $S$ and contributions from each other connected component (all of which are isolated points). Theorem \ref{thmconics_tgt_5_conics} shows that the contribution from $S$ is $\frac{6^5 - 3264}{2}(\lra{1} + \lra{-1})$, which is independent of the choice of conics. 

The component $Z_S$ of  $\{\sigma = 0\}$ supported on $S$ is non-reduced, however (corresponding to the $\calO(-2E)$ term in the proper transfer above) \cite[p 463(d)]{EisenbudHarris}. The closed immersion $Z_S \hookrightarrow \bbb{P}^5_k$ is not a regular embedding, preventing us from defining an excess bundle and applying \cite[Theorem p.2]{Fasel-excess} \cite[Proposition 3.3.4]{DJK} or \cite[Remark 5.22]{BW-A1Eulerclasses}.
\end{remark}

We turn to the development of tools for computing $\mathbb{A}^1$-excess and residual intersections without assuming the existence of an excess bundle or extensive knowledge of the particular geometry.

\section{Complements on residual intersections} \label{sec:res-int} 
We recall, strengthen and generalize some results of Eisenbud--Ulrich \cite{eisenbud2019duality} on canonical modules in residual intersections.
\subsection{Introduction}
Throughout this section, except for \S\S\ref{subsec:modified-koszul},\ref{subsec:globalization}, we fix a noetherian local Gorenstein ring $R$ and an ideal $I$ of codimension $g$.
\subsubsection{Residual intersections}
Let $J \subset I$ be an ideal which can be generated by $s \ge g$ elements.
Then $K = J:I$ is called an \emph{$s$-residual intersection} if $\codim K \ge s$.
The residual intersection is called \emph{geometric} if $\codim K+I \ge s+1$.
We put $t = s-g \ge 0$.
The case $t=0$ is called \emph{linkage} \cite{peskine1974liaison}.

\subsubsection{Geometric interpretation} \label{subsec:quotient-geometric}
Recall that we have an equality of sets \[ Z(J) = Z(K) \cup Z(I) \subset \Spec(R) \] (e.g. apply \cite[Tags 00L3(4) and 00L2]{stacks-project} to $0 \to I/J \to R/J \to R/I \to 0$) and so in particular \[ Z(K) \supset \overline{Z(J) \setminus Z(I)}. \]
One may prove that equality holds if $J$ is reduced, but not in general \cite[Chapter 4 \S4 Theorem 7 ]{CoxLittleOshea-IdealsVarietiesAlgorithms} (the hypothesis that $k$ be algebraically closed here is not necessary because we use schemes not varieties).\NB{Let $h \in I(V(J) \setminus V(I))$, i.e. $h \in P$ whenever $J \subset P$ but $I \not\subset P$. Let $g \in I$. Then $hg \in P$ whenever $J \subset P$. Hence $hg \in \sqrt{J}$, and so $h \in \sqrt{J}:I$. Thus $Z(\sqrt{J}:I) \subset \overline{Z(J) \setminus Z(I)}$.}
We will constantly use this geometric interpretation.
One consequence is that if $\codim I > \codim J$ then $\codim K = \codim J$.\NB{We have $W=\overline{Z(J) \setminus Z(I)} \subset Z(K) \subset Z(J)$. It thus suffices to show that $\codim W \le \codim J$. If $P \supset J$ has codimension $c$, then either $Z(P) \subset W$ and so $\codim W \le c$, or $P \supset I$, whence $\codim I \le c$. Picking $P$ appropriately we get $c=\codim J < \codim I$, whence the claim.}

\subsubsection{Standard hypotheses} \label{subsub:hypotheses}
In order to control the situation, one puts conditions on $I$. Denote by $\mu(I)$ the minimal number of generators of an ideal $I$.
We say that $I$ satisfies the condition $G_s$ if \[ \mu(I_P) \le \codim P \] for all prime ideals $P \supset I$ with $\codim P \le s-1$ \cite[\S1]{eisenbud2019duality}.
We say that $I$ satisfies the condition $D_t$ if \[ \depth(R/I^j) \ge \dim(R/I) - j + 1 \] for all $j \le t$.

\begin{definition}
In the language of \cite[\S1]{eisenbud2019duality}, for fixed $s$ and $t = s-\codim I$, $I$ satisfies the weak, standard or strong hypothesis if $I$ satisfies $G_s$ and in addition $D_{t-1}$, $D_t$ or $D_{t+1}$.
\end{definition}
\begin{remark}
Observe that $G_s$ implies $G_{s-1}$ and $D_t$ implies $D_{t-1}$.
In particular the strong hypothesis implies the standard hypothesis, and so on, justifying the terminology.
\end{remark}

For us, the main utility of the strong hypothesis is that it ensures that if $K=J:I$ is an $s$-residual intersection then $R/K$ is Cohen--Macaulay.
Moreover then $I^{t+1}/JI^t$ is a canonical module for $R/K$.
See \cite[Theorem 3.1]{eisenbud2019duality}.
Similarly if $I$ satisfies the standard hypothesis, then it satisfies the strong hypothesis with respect to $s-1$, and hence any $(s-1)$-residual intersection $K_{s-1} = J_{s-1}:I$ is also Cohen--Macaulay.

\subsubsection{Good generators} \label{subsub:good-gens}
Let $J=(a_1, \dots, a_s)$ be a sequence of generators of $J$.
We put $J_i = (a_1, \dots, a_i)$ and $K_i = J_i:I$.
We call the generators \emph{good} if $K_i$ is a geometric $i$-residual intersection for all $g \le i < s$.
Note in particular that since $I$ has codimension $g$ and $K_g$ has codimension $g$, also $J_g$ must have codimension $g$ (see the end of \S\ref{subsec:quotient-geometric}); in other words $(a_1, \dots, a_g)$ is a regular sequence \cite[Theorem 2.1.2(c)]{bruns1998cohen}.

If $I$ satisfies $G_s$ then good generators always exist \cite[Corollary 1.6(a)]{ulrich1994artin}.

\subsection{Freeness of the conormal module}

\begin{proposition} \label{prop:conormal-free}
Let $K=J:I$ be an $s$-residual intersection in the local noetherian Gorenstein ring $R$.
Suppose that $I$ satisfies the weak hypothesis.
Then $J/KJ$ is free over $R/K$.
In fact if $a_1, \dots, a_s$ generate $J$, then $J/KJ$ is free on their images.
\end{proposition}
\begin{proof}
Generators $a_1, \dots, a_s$ yield a surjection $(R/K)^s \to J/KJ$.
If $J/KJ \wequi (R/K)^s$ then the map is necessarily an isomorphism \cite[Theorem 2.4]{Matsumura-CommutativeAlg}; hence we may pick generators conveniently.

Let $a_1, \dots, a_s$ be good generators (see \S\ref{subsub:good-gens}).
We shall show by induction on $u$ that $J_u/K_uJ_u$ is free on the images of the $a_j$; the desired result is the case $u=s$.
We must thus show that whenever \[ x_1 a_1 + \dots + x_u a_u =: r \in K_u J_u \] then $x_j \in K_u$ for all $j$.
By assumption we have $r = w_1 a_1 + \dots + w_u a_u$, with $w_j \in K_u$.
To prove that $x_j \in K_u$ we may as well show that $x_j - w_j \in K_u$; hence we may replace $x_j$ by $x_j - w_j$ and so assume that $r=0$.
By \cite[Proposition 3.3(1)]{eisenbud2019duality}, $a_u$ is regular on $R_{u-1} = R/K_{u-1}$, in the sense that $a_u$ is not a zero-divisor in $R_{u-1}$.
Since \[ x_u a_u = -(x_1 a_1 + \dots + x_{u-1} a_{u-1}) \in J_{u-1} \subset K_{u-1}, \] we have $x_u a_u = 0$ in $R_{u-1}$, whence $x_u \in K_{u-1}$ by regularity of $a_u$.
Let $i$ be an element of $I$.
Then $x_u i \in K_{u-1}I \subset J_{u-1}$, so $x_ui = y_1 a_1 + \dots + y_{u-1}a_{u-1}$.
Multiplying the original equation by $i$ and substituting our expression for $x_u i$ we get \[ (ix_1 + a_uy_1)a_1 + \dots + (ix_{u-1}+a_uy_{u-1})a_{u-1} = 0. \]
By induction, the coefficients $ix_j + a_u y_j$ are in $K_{u-1}$, so that $ix_j \in K_{u-1}+(a_u)$. Since $i$ was arbitrary, it follows that $x_j \in (K_{u-1} + (a_u)) : I$.  By \cite[Proposition 3.3(1)]{eisenbud2019duality}, $$K_u = (K_{u-1}+(a_u)): (K_{u-1} + I) .$$ From the definition of ideal quotients, we have $ (K_{u-1}+(a_u)): (K_{u-1} + I) =(K_{u-1} + (a_u)) : I$. Thus $K_u = (K_{u-1} + (a_u)) : I$ and $x_j \in K_u$ as desired. 
\end{proof}

\subsection{Modified Koszul complexes} \label{subsec:modified-koszul}
In anticipation of globalizing our results in the last subsection, we formulate the following in a slightly more general form than strictly needed right now.
\begin{definition} \label{def:K'}
Let $X$ be a scheme, $V$ a vector bundle on $X$, $\sigma$ a section of $V$, $$J = \image(\sigma^\dual: V^\dual \to \scr O_X),$$ $I \supset J$ an ideal sheaf and $t \ge -1$.
Denote by $\Kos(V,\sigma)_\bullet$ the usual Koszul complex, so that $\Kos(V, \sigma)_n = \Lambda^n V^\dual$. See for example, \cite[Tag 0622]{stacks-project}.
Let \[ \Kos'(V,\sigma)_n = I^{t+1-n} \Kos(V, \sigma)_n \subset \Kos(V, \sigma)_n, \] where $I^i := \scr O_X$ for $i \le 0$.
This is a complex (because the differential of $\Kos(V, \sigma)_\bullet$ is essentially given by multiplication by $\sigma^\dual$), which we call the \emph{modified Koszul complex}.
If $V$ has rank $s$, projection to the highest term defines a canonical map \[ \alpha_{V, \sigma}: \Kos'(V, \sigma) \to \Kos(V, \sigma) \to \Sigma^s \det V^\dual. \]
By construction $\pi_0 \Kos'(V,\sigma) \wequi I^{t+1}/JI^t$ (and $\pi_i \Kos'(V, \sigma) = 0$ for $i < 0$) and hence we have another canonical map \[ \Kos'(V, \sigma) \to I^{t+1}/JI^t. \]
\end{definition}

\begin{definition} \label{def:K'-aff}
Taking $V = \scr O_X^s$, $\sigma = (a_1, \dots, a_s)$ corresponds to a choice of $s$ generators of $J$.
In this case we write \[ \Kos'(a_1, \dots, a_s) := \Kos'(V, \sigma) \] and denote the map $\alpha_{V, \sigma}$ by $\alpha_{a_1, \dots, a_s}$.
\end{definition}

\begin{remark} \label{rmk:K'-functorial}
For fixed $t$, $\Kos'$ is functorial in the data $(V,\sigma,I)$.
That is, given another vector bundle $V'$ and a morphism $f: V'^\dual \to V^\dual$, we obtain $f_*: \Kos(V', f^\dual \sigma) \to \Kos(V, \sigma)$ given in degree $n$ by $\Lambda^n f$ \cite[Tag 0624]{stacks-project}.
This restricts to $f_*: \Kos'(V',  f^\dual \sigma) \to \Kos'(V, \sigma)$.
\end{remark}

\subsection{Canonical traces}\NB{Justification of the term ``trace:" The counit of adjunction $f_!f^! \scr O \to \scr O$ is often called the trace. Taking $f=i$, we construct a map $\tau: i_! M \to \scr O[s]$ which induces by adjunction an isomorphism $\tau^\dagger: M \to i^! \scr O[s]$. Under this isomorphism, $\tau$ corresponds to the counit. }
Fix an $s$-residual intersection $K=J:I$ satisfying the strong hypothesis.
We work in the derived category of $R$.
Write $\Dual(\ph) = \RiHom(\ph, R)$ for the standard duality.
Since $R$ is Gorenstein, $R/K$ is Cohen--Macaulay of codimension $s$ (see \S\ref{subsub:hypotheses}), and $I^{t+1}/JI^t$ is a canonical module for $R/K$ \cite[Theorem 3.1]{eisenbud2019duality}, there exists a \emph{non-canonical} isomorphism \cite[Proposition 3.3.3(bi), Theorem 3.3.4(b) and Theorem 3.3.7]{bruns1998cohen} \[ \Sigma^s \Dual(R/K) \wequi I^{t+1}/JI^t. \]
\begin{remark}
In fact, given a choice of good generators of $J$, \cite[Theorem 4.1]{eisenbud2019duality} exhibits a specific such isomorphism.
One easily sees that this depends on the choices; for example multiplying any one of the generators by a unit of $R$ changes the isomorphism by that unit.
\end{remark}

To obtain a canonical isomorphism, we shall tensor the right hand side by the line bundle $\det(J/KJ)^\dual$; this makes sense by Proposition \ref{prop:conormal-free}.
Before we can do so, we need the following result.
\begin{lemma} \label{lemm:I-resn}
Suppose that $I$ satisfies the standard hypothesis and $J=(a_1, \dots, a_s)$.
Then the map $\Kos'(a_1, \dots, a_s) \to I^{t+1}/JI^t$ (constructed in Definition \ref{def:K'}) is an equivalence.
\end{lemma}
\begin{proof}
We first show that this assertion is independent of the choice of generators.
By Remark \ref{rmk:K'-functorial}, given a free finitely generated module $F$ and a surjection $\sigma^\dual: F^\dual \to J$, we functorially obtain a complex $\Kos'(F, \sigma)$ and a morphism $\Kos'(F,\sigma) \to I^{t+1}/JI^t$.
Now let $(b_1, \dots, b_s)$ be another set of generators of $J$.
Writing the $b_i$ in terms of the $a_j$ we obtain a (perhaps non-unique) matrix $M: R^s \to R^s$ such that $\sigma^\dual_a \circ M = \sigma^\dual_b$.
Thus there is an induced morphism $\Kos'(M): \Kos'(R^s, \sigma_b) \to \Kos'(R^s, \sigma_a)$ over $I^{t+1}/JI^t$.
The two maps $\sigma^\dual_a, \sigma^\dual_b: R^s \to J$ become isomorphisms modulo $K$ by Proposition \ref{prop:conormal-free}, so that also $M$ must be an isomorphism modulo $K$ as well.
Thus in particular $M$ is surjective ($R$ being local\NB{Nakayama's lemma}), and hence an isomorphism \cite[Theorem 2.4]{Matsumura-CommutativeAlg}.
Consequently the desired statement for the two sets of generators is equivalent.

We may thus assume that the generators $(a_1, \dots, a_s)$ are good.
Consider the graded ring \[ S = \bigoplus_{j \in \Z} S^{(j)}, S^{(j)} := I^j, \] where $I^j := R$ for $j \le 0$.
We have elements $a_1, \dots, a_s \in S^{(1)}$.
The grading on $S$ supplies the Koszul complex $\Kos_S(a_1, \dots, a_s)$ with an additional grading, i.e. a decomposition of complexes \[ \Kos_S(a_1, \dots, a_s) = \bigoplus_i \Kos_S(a_1, \dots, a_s)^{(i)}, \quad \Kos_S(a_1, \dots, a_s)^{(i)}_n = (\Lambda^n S^s)^{(i-n)}. \]
We have $\Kos_S(a_1, \dots, a_s)^{(t+1)} = \Kos'(a_1, \dots, a_s)$.
We shall prove by induction on $0 \le i \le s$ that for $i-g+1 \le j \le t+1$ we have $\Kos_S(a_1, \dots, a_i)^{(j)} \wequi I^j/J_iI^{j-1}$. Since $$\pi_0(\Kos_S(a_1, \dots, a_i)) \cong S/J_i,$$ this is equivalent to the claim that for such $(i,j)$, the complex $\Kos_S(a_1, \dots, a_i)^{(j)}$ has homology groups concentrated in degree $0$.

If $i=0$, then $\Kos_S(a_1,\ldots,a_i) \cong S$ and the complex $\Kos_S(a_1,\ldots,a_i)^{(j)} $ is concentrated in degree $0$ for all $j$.
Assume the claim holds for $i-1$, and let $i-g+1 \le j \le t+1$.
We have a cofiber sequence \[ \Kos_S(a_1, \dots, a_{i-1})^{(j-1)} \xrightarrow{a_{i}} \Kos_S(a_1, \dots, a_{i-1})^{(j)} \to \Kos_S(a_1, \dots, a_i)^{(j)}. \]
By induction, the complexes $\Kos_S(a_1, \dots, a_{i-1})^{(j-1)}$ and $\Kos_S(a_1, \dots, a_{i-1})^{(j)}$ are concentrated in degree $0$, so the first map is equivalent to $$I^{j-1}/J_{i-1}I^{j-2} \xrightarrow{a_{i}} I^{j}/J_{i-1}I^{j-1},$$ and it suffices to show this map is injective. 

For $j \le 1$, we have $I^{j-1}/J_{i-1}I^{j-2} = R/J_{i-1}$ and $I^{j}/J_{i-1}I^{j-1} \subseteq R/J_{i-1}$. It thus suffices to prove that the composite of multiplication by $a_{i}$ \[ R/J_{i-1} \to I/J_{i-1} \subset R/J_{i-1} \] is injective. Since $i-g+1\le j \le 1$ we have that $i \le g$. But $(a_1, \dots, a_g)$ is a regular sequence (see \S\ref{subsub:good-gens}), so we are done in this case.

If $j \ge 2$, the induction step is provided by \cite[Lemma 2.7(a)]{ulrich1994artin}, applied with $r=t$ (the condition $AN_{s-1}^-$ \cite[Definition 1.2]{ulrich1994artin} is implied by $AN_{s-1}$ which follows from the standard hypothesis; see \S\ref{subsub:hypotheses}).
\end{proof}

Here is the main result of this section.
\begin{theorem} \label{thm:canonical-trace}
Let $K=J:I$ be an $s$-residual intersection (in a noetherian local Gorenstein ring $R$) satisfying the strong hypothesis.
Pick generators $J=(a_1, \dots, a_s)$.
The following hold.
\begin{enumerate}
\item \label{tmap_independent_generators} The composite \[ \tmap: I^{t+1}/JI^t \otimes_{R/K} \det(J/KJ)^\dual \wequi I^{t+1}/JI^t \wequi \Kos'(a_1, \dots, a_s) \xrightarrow{\alpha_{a_1, \dots, a_s}} \Sigma^s R \] (where the first equivalence is via $da_1 \wedge \dots \wedge da_s \in \det(J/KJ)^\dual$ and the second is by Lemma \ref{lemm:I-resn}; for $\alpha_{a_1, \dots, a_s}$ see Definition \ref{def:K'-aff}) is independent of the choice of generators.

\item \label{tmap_dag_iso} The map \[ \tmap^\dagger: I^{t+1}/JI^t \otimes_{R/K} \det(J/KJ)^\dual \to \Sigma^s \Dual(R/K) \] adjoint to the composite \begin{gather*} R/K \otimes_R^L (I^{t+1}/JI^t \otimes_{R/K} \det(J/KJ)^\dual) \to R/K \otimes_R (I^{t+1}/JI^t \otimes_{R/K} \det(J/KJ)^\dual) \\ \xrightarrow{\text{mult}} I^{t+1}/JI^t \otimes_{R/K} \det(J/KJ)^\dual \xrightarrow{\tmap} \Sigma^s R \end{gather*} is an isomorphism.
\end{enumerate}
\end{theorem}
\begin{remark} \label{rmk:uniqueness-meaning}
\begin{enumerate}
\item Recall that given an $\infty$-category $\scr C$ and objects $X, Y \in \scr C$, we obtain a \emph{space} $\Map(X, Y)$.
  When working with $\infty$-categories, any sensible construction should produce a point in the space $\Map(X, Y)$ and not in the set $\pi_0\Map(X,Y)$.
  Similarly, when comparing two constructions, one should not just show that two points of $\Map(X, Y)$ have the same image in $\pi_0\Map(X,Y)$, but in addition provide a \emph{specific path} between the two points.
  This additional bookkeeping is how $\infty$-categories overcome problems that plague derived categories (such as non-uniqueness of cones).
\item There is one case in which ``doing the right thing'' (providing points and paths in $\Map(X,Y)$) is the same as ``doing the old thing'' (providing points of and checking equalities in $\pi_0\Map(X,Y)$), namely when $\Map(X,Y)$ is discrete (i.e. $\pi_i(\Map(X,Y), f) = 0$ for all $i>0, f\in \Map(X,Y)$; in other words $\Map(X,Y) \to \pi_0\Map(X,Y)$ is an equivalence of spaces).
\item Theorem \ref{thm:canonical-trace}\ref{tmap_independent_generators}  asserts that a certain construction is independent of choices.
  A priori this at least means that for any two choices we should provide an explicit homotopy connecting them; even better we should prove that the space of such homotopies is contractible.
  Luckily the space of maps from $I^{t+1}/JI^t \otimes_{R/K} \det(J/KJ)^\dual$ to $\Sigma^s R$ is equivalent (perhaps non-canonically, before the theorem is proved) to \[ \Map(R, \Sigma^s \Dual(I^{t+1}/JI^t \otimes_{R/K} \det(J/KJ)^\dual)) \wequi \Map(R, \Dual^2(R/K)) \wequi R/K \] and so in particular is discrete.
  Thus the space of homotopies is contractible if and only if it is nonempty, that is, for any two choices, the resulting maps are homotopic.
  In other words establishing the strongest possible meaning of ``independent of choices'' is equivalent to the weakest possible meaning in this case.
\end{enumerate}
\end{remark}

\begin{proof}[Proof of Theorem \ref{thm:canonical-trace}.]
We first show \ref{tmap_independent_generators}, i.e. that $\tmap$ is independent of the choice of generators.
Thus let $(a_1, \dots, a_s)$ and $(a_1', \dots, a_s')$ be two such choices.
As in the proof of Lemma \ref{lemm:I-resn}, writing the $a_i'$ in terms of the $a_i$ we obtain from Remark \ref{rmk:K'-functorial} an invertible matrix $M: R^n \to R^n$ inducing $\Kos'(M): \Kos'(a_1', \dots, a_s') \xrightarrow{\wequi} \Kos'(a_1, \dots, a_s)$.
Put $d=\det M \in R^\times$.
The map $R \wequi \Kos(a_1', \dots, a_s')_s \to \Kos(a_1, \dots, a_s)_s \wequi R$ is given by multiplication by $d$, and hence the same is true for the restriction to the subcomplexes $\Kos'$.
Consequently the left hand square of the following diagram commutes; the right hand square also commutes essentially by construction
\begin{equation*}
\begin{CD}
\Sigma^s R @<{\alpha_{a_1', \dots, a_s'}}<< \Kos'(a_1', \dots, a_s') @>{\wequi}>> I^{t+1}/JI^t \otimes_{R/K} \det(J/KJ)^\dual \\
@V{d}VV                                       @V\Kos'(M)VV                                             @VdVV \\
\Sigma^s R @<{\alpha_{a_1, \dots, a_s}}<< \Kos'(a_1, \dots, a_s) @>{\wequi}>> I^{t+1}/JI^t \otimes_{R/K} \det(J/KJ)^\dual.
\end{CD}
\end{equation*}
The unlabelled horizontal maps are the canonical equivalences (of Lemma \ref{lemm:I-resn}), twisted by the determinant of the duals of the generators.
If $\tmap$ and $\tmap'$ denote the relevant composites we seek to prove are homotopic, we thus learn that $d\tmap' = \tmap d$.
The result follows since $d$ is a unit.\NB{And $\Dcat(R)$ is $R$-linear...}

To prove \ref{tmap_dag_iso}, i.e. that $\tmap^\dagger$ is an equivalence, we may thus choose generators conveniently, and so assume they are good (see \S\ref{subsub:good-gens}).
Also the twist by $\det(J/KJ)^\dual$ is irrelevant for this problem.
We prove by induction on $i$, for $g \le i \le s$, that $\tmap_{a_1, \dots, a_{i}}^\dagger: I^{i-g+1}/J_iI^{i-g} \to \Sigma^i \Dual(R/K_i)$ is an isomorphism.

Consider first the case $i=g$.
The commutative diagram
\begin{equation*}
\begin{tikzcd}
I/J_g \wequi \Kos'(a_1, \dots, a_g) \ar[d] \ar[rd, "\alpha_{V,\sigma}"] \\
R/J_g \wequi \Kos(a_1, \dots, a_g) \ar[r] & \Sigma^g R
\end{tikzcd}
\end{equation*}
(the bottom left hand equivalence is because $a_1, \dots, a_g$ is a regular sequence (see \S\ref{subsub:good-gens}) and the top left hand equivalence is by Lemma \ref{lemm:I-resn}) induces by passing to adjoints (using Lemma \ref{lemm:adjoints} below\NB{$P = \Sigma^gR$, $A=R/J_g$, $B=R/K_g$, $M=K'$, $N=K$}) a commutative diagram
\begin{equation} \label{eq:can-proof-1}
\begin{tikzcd}
I/J_g \wequi \Kos'(a_1, \dots, a_g) \ar[d] \ar[r] & \Sigma^g \Dual(R/K_g) \ar[d, "\iota"] \wequi \omega_{R/K_g} \\
R/J_g \wequi \Kos(a_1, \dots, a_g) \ar[r, "\sigma"] & \Sigma^g \Dual(R/J_g) \wequi \omega_{R/J_g}.
\end{tikzcd}
\end{equation}
(E.g. the map $I/J_g \to \Sigma^g \Dual(R/K_g)$ is equivalent by adjunction to a map $I/J_g \otimes^L R/K_g \to \Sigma^g R$, which we obtain from the $R/K_g$-module structure on $I/J_g$ and the map $I/J_g \to \Sigma^g R$.)
Here the map $\iota$ is the dual of $R/J_g \to R/K_g$ and the left hand vertical map is the canonical inclusion.
Noting that $I/J_g = I \cdot R/J_g$, it suffices to prove that $\sigma$ is an isomorphism and $\iota$ is an injection with image $I \cdot \omega_{R/J_g}$.
Unwinding the definitions, the map $\sigma$ identifies with the canonical self-duality of $\Kos(a_1, \dots, a_g)$\NB{Compute $\Dual(R/J_g)$ in terms of the free resolution $\Kos(a_1, \dots, a_g)$.} \cite[Proposition 17.15]{Eisenbud_CommutativeAlgebra}, whence the former statement holds.
The latter statement is \cite[Theorem 4.1(1)]{eisenbud2019duality}.

Now suppose the result has been proved for $i-1$.
Let $K''$ be the mapping cone on $\Kos'(a_1, \dots, a_{i-1}) \xrightarrow{a_i} \Kos'(a_1, \dots, a_{i-1})$, and similarly let $R''$ be the mapping cone on $R/K_{i-1} \xrightarrow{a_i} R/K_{i-1}$.
We have a canonical inclusion $\Kos'(a_1, \dots, a_i) \subset K''$.
Since $a_i$ is regular on $R/K_{i-1}$ \cite[Proposition 3.3(1)]{eisenbud2019duality}, $R'' \wequi R/(K_{i-1}+(a_i))$ is a ring and $K''$ is an $R''$-module; in fact $K''$ is a canonical module for $R''$ (and so in particular concentrated in degree $0$), since $\Kos'(a_1, \dots, a_{i-1})$ is for $R/K_{i-1}$ by induction (use \cite[Theorem 3.3.5]{bruns1998cohen}).
The commutative diagram
\begin{equation*}
\begin{tikzcd}
K'' \ar[r, "\partial"] & \Sigma \Kos'(a_1, \dots, a_{i-1}) \ar[r, "\Sigma\alpha_{i-1}"] & \Sigma^i R \\
\Kos'(a_1, \dots, a_i) \ar[u] \ar[rru, "\alpha_i" swap]
\end{tikzcd}
\end{equation*}
induces by adjunction (using Lemma \ref{lemm:adjoints} below\NB{$P = \Sigma^iR$, $A=R''=R/(K_{i-1}+(a_i))$, $B=R/K_i$, $M=K'$, $N=K''$}) a commutative diagram
\begin{equation} \label{eq:can-proof-2}
\begin{CD}
K'' @>{(\tmap_{i-1}^\dagger)'}>> \Sigma^i \Dual(R'') \wequi \omega_{R''} \\
@A{\iota'}AA            @A{\iota}AA     \\
\Kos'(a_1, \dots, a_{i}) @>{\tmap_{i}^\dagger}>> \Sigma^i \Dual(R/K_{i}) \wequi \omega_{R/K_{i}}.
\end{CD}
\end{equation}
The map $(\tmap_{i-1}^\dagger)'$ sits in a diagram of cofiber sequences
\begin{equation*}
\begin{CD}
\Kos'(a_1, \dots, a_{i-1}) @>{a_i}>> \Kos'(a_1, \dots, a_{i-1}) @>>> K'' \\
@V{\tmap_{i-1}^\dagger}VV          @V{\tmap_{i-1}^\dagger}VV @V(\tmap_{i-1}^\dagger)'VV \\
\Sigma^i \Dual(R/K_{i}) @>{a_i}>> \Sigma^i \Dual(R/K_{i}) @>>> \Sigma^i \Dual(R'')
\end{CD}
\end{equation*}
and hence is an isomorphism.
The map $\iota$ is dual to the projection $R'' \to R/K_{i}$ and is an injection with image $I \cdot \omega_{R''}$ \cite[Corollary 4.3]{eisenbud2019duality}.
The map $\iota'$ identifies with the natural map \[ I^{i-g+1}/J_iI^{i-g} \to I^{i-g}/(J_{i-1}I^{i-g-1}+a_iI^{i-g}) =: M \]
Its image is $I \cdot M$, and it is injective because \begin{gather*} I^{i-g+1} \cap (J_{i-1}I^{i-g-1}+a_iI^{i-g}) = I^{i-g+1} \cap J_{i-1}I^{i-g-1} + a_iI^{i-g} \\ \subset I^{i-g+1} \cap K_{i-1} + a_iI^{i-g} = J_iI^{i-g-1}, \end{gather*} where the last equality is by \cite[Proposition 3.3(5)]{eisenbud2019duality}.
Hence $\tmap_i$ is an isomorphism as desired.

\end{proof}

We made use of the following general $1$-categorical construction (which we apply to the homotopy category of $\Dcat(R)$).
For the notion of commutative algebra in a symmetric monoidal category, and categories of modules thereover, see \cite[4.1]{Brandenburg14} or \cite[1.2 and 1.3, Def 1.2.8, Def 1.2.10]{MartyF09}.
\begin{lemma}[Passage to adjoints] \label{lemm:adjoints}
Let $\scr C$ be a closed symmetric monoidal category with tensor product $\otimes$ and internal hom $\iHom$, $A \to B \in \CAlg(\scr C)$, $M$ a $B$-module, $N$ an $A$-module, and $P \in \scr C$.
Suppose given an $A$-module morphism $M \to N$ and a morphism $N \to P$.
Then the commutative diagram
\begin{equation*}
\begin{tikzcd}
M \ar[d] \ar[dr] \\
N \ar[r] & P
\end{tikzcd}
\end{equation*}
induces a canonical commutative diagram
\begin{equation*}
\begin{CD}
M @>>> \iHom(B, P) \\
@VVV        @VVV   \\
N @>>> \iHom(A, P).
\end{CD}
\end{equation*}
\end{lemma}
\begin{proof}\NB{In element notation, this statement is completely trivial...}
Expand the diagram as
\begin{equation*}
\begin{tikzcd}
                 & \iHom(B, M) \ar[r] \ar[d] & \iHom(B,P) \ar[dd] \\
M \ar[ru] \ar[r] \ar[d] & \iHom(A, M) \ar[d] \\
N \ar[r]                & \iHom(A, N) \ar[r] & \iHom(A, P).
\end{tikzcd}
\end{equation*}
All cells commute for trivial reasons.
\end{proof}

\subsection{Aside: explicit duality} \label{subsec:explicit}
The canonical trace maps of Theorem \ref{thm:canonical-trace} can be made more explicit, at least in some cases.
These results are not used in the sequel, and so the reader may wish to skip this subsection.
In what follows, we give an explicit formula for the trace map in a \emph{geometric} residual intersection, generalizing the work of Scheja--Storch \cite{scheja}.

\begin{remark}
Without the geometricity assumption, the most optimistic potential formulas for the trace forms are false.
This is discussed (with slightly different emphasis) in \cite[\S7]{eisenbud2019duality}.
\end{remark}

\begin{lemma} \label{lemm:explicit}
Notation and assumptions as in Theorem \ref{thm:canonical-trace}.
The following two maps are homotopic:
\begin{gather*} \Sigma^s \Dual(R/K) \otimes \det(J/KJ) \xrightarrow{(\tmap^\dagger)^{-1}} I^{t+1}/JI^t \otimes \det(J/KJ)^\dual \otimes \det(J/KJ) \\ \wequi I^{t+1}/JI^t \to R/J \\ \text{and} \\ \Sigma^s \Dual(R/K) \otimes \det(J/KJ) \stackrel{(1)}{\wequi} \Sigma^s \Dual(R/K) \xrightarrow{(2)} \Sigma^s \Dual \Kos(a_1, \dots, a_s) \\ \stackrel{(3)}{\wequi} \Kos(a_1, \dots, a_s) \xrightarrow{(4)} R/J. \end{gather*}
The last map of the first composite is induced by $I^{t+1} \to R \to R/J$.
In the second composite, $(1)$ trivializes $\det(J/KJ)$ using the generators $(a_1, \dots, a_s)$ of $J/KJ$, $(2)$ is dual to $\Kos(a_1, \dots, a_s) \to R/J \to R/K$, $(3)$ is the canonical self duality of the Koszul complex and $(4)$ is $\Kos(a_1, \dots, a_s) \to R/J$.

If the residual intersection is geometric, these maps (of $R$-modules concentrated in homotopical degree $0$) are injective.
\end{lemma}
\begin{proof}
Note that the first map does not depend on the choice of generators.
We first show the same holds for the second.
If $a_1', \dots, a_s'$ is a second set of generators of $J$, then writing the $a_i' = \sum_j M_{ij} a_j$ we obtain a map $f=\Kos(M): \Kos(a_1, \dots, a_s) \to \Kos(a_1', \dots, a_s')$.
The pairing \[ \Kos(a_1, \dots, a_s) \otimes \Kos(a_1, \dots, a_s) \xrightarrow{f \otimes f} \Kos(a_1', \dots, a_s') \otimes \Kos(a_1', \dots, a_s') \xrightarrow{m} \Sigma^s R \] is readily verified to be $\det(M)\cdot m$, where $m$ is the usual pairing (i.e. multiplication followed by projection).
This implies that the two maps \[ \Sigma^s \Dual(R/K) \to \Sigma^s \Dual \Kos(?) \wequi K \to R/J, \] where $\Kos(?)$ is either $\Kos(a_1,\dots,a_s)$ or $\Kos(a_1', \dots, a_s')$, differ precisely by $\det(M)$.
Using this one easily checks that when incorporating the twist by $\det(J/KJ)$, the second map in the statement indeed becomes independent of the choice of generators.

It follows that we may assume the generators are good, and we may ignore the twists.
By induction, we shall prove the statement for the ideal $J_i = (a_1, \dots, a_i)$, $g \le i \le s$.
In the case $i=g$, Diagram \eqref{eq:can-proof-1} precisely shows what we want (the two composites in question are obtained by traversing Diagram \eqref{eq:can-proof-1} from top right hand corner to bottom left hand corner in the two possible ways, inverting the horizontal isomorphisms).

Now we prove by induction on $i>g$ the remaining cases.
Consider the diagram (in which $/a_i$ means cofiber of multiplication by $a_i$)
\begin{equation*}
\begin{tikzcd}
\Sigma^i \Dual \Kos(a_1, \dots, a_i) \ar[r, equal] & \Sigma^i \Dual (\Kos(a_1,\dots,a_{i-1})/a_i) \ar[dr] \\
\Sigma^i \Dual(R/K_i) \ar[r] \ar[u] \ar[d] & \Sigma^i\Dual((R/K_{i-1})/a_i) \ar[u]\ar[d] & R/(J_{i-1}+(a_i)) = R/J_i \\
I^{i-g+1}/J_iI^{i-g} \ar[r] & I^{i-g}/(I^{i-g-1}+a_i I^{i-g}) \ar[ru].
\end{tikzcd}
\end{equation*}
The upper square commutes, the lower square commutes essentially by construction (see also Diagram \eqref{eq:can-proof-2}), and the triangle commutes by induction (it is obtained from the commuting triangle for $i-1$ by dividing out $a_i$).
This implies commutativity of the two longest paths, which is the desired statement.

Finally we prove the injectivity statement in the case of geometric residual intersections.
By \cite[Proposition 3.3(5)]{eisenbud2019duality} we have $I^{t+1} \cap J \subset I^{t+1} \cap K = JI^t$, which implies what we want.
\end{proof}

\begin{remark}
In the case of a geometric residual intersection, Lemma \ref{lemm:explicit} supplies an alternative construction of the canonical trace.
(Restrict the codomain of the second map to its image to obtain an isomorphism).
\end{remark}

We denote by \[ \bar\tmap: \Sigma^s \Dual(R/K) \to R/J \] the map considered in Lemma \ref{lemm:explicit} (trivializing the twist for a given choice of generators).
For modules $M, N$ over a complete local ring $R$, we denote by $M \hat\otimes N$ the completed tensor product.
\begin{proposition} \label{prop:explicit-formula}
Let $K = J:I$ be an $s$-residual intersection satisfying the strong hypothesis.
Suppose in addition that $R = A[[T_1, \dots, T_s]]$, $A$ is regular local and $A \to R/K$ is finite.
Pick generators $J=(a_1, \dots, a_s)$.

The kernel of $R\, \hat\otimes_A R \to R$ is generated by $T_i \otimes 1 - 1 \otimes T_i$ (as an ideal of $R\, \hat\otimes_A R$), and contains $a_i \otimes 1 - 1 \otimes a_i$.
Write \[ a_i \otimes 1 - 1 \otimes a_i = \sum_j b_{ij}(T_i \otimes 1 - 1 \otimes T_i), \] for some $b_{ij} \in R\, \hat\otimes_A R$.
Let \[ \bar \Delta \in R/K \otimes_A R/J \] be the image of $\det(b_{ij})$.
Write \[ \bar \Delta = \sum_k r_k \otimes i_k, \] with $r_k \in R/K$ and $i_k \in R/J$.
\begin{enumerate}
\item \label{it:ex_dual_1} The morphism \[ \beta: \Hom_A(R/K, A) \wequi \Sigma^s \Dual(R/K) \xrightarrow{\bar\tmap} R/J \] is given by \begin{equation} \label{eq:explicit-duality} \varphi \mapsto \sum_k \varphi(r_k) i_k. \end{equation}
\item \label{it:ex_dual_2} If the residual intersection is geometric and the generators are good, then $\bar \Delta$ is contained in the image of (the injective map) \[ R/K \otimes_A I^{t+1}/JI^t \to R/K \otimes_A R/J. \]
  In particular one may choose $i_k \in I^{t+1}/JI^t$ and the isomorphism \[ \Hom_A(R/K, A) \wequi \Sigma^s \Dual(R/K) \xrightarrow{(\tmap^\dagger)^{-1}} I^{t+1}/JI^t \] is given by \eqref{eq:explicit-duality}.
\end{enumerate}
\end{proposition}
\begin{proof}
Put $B := R/K$.
Firstly, note that $B$ is flat over $A$ by \cite[Theorem 18.16]{Eisenbud_CommutativeAlgebra} (recall that $R/K$ is Cohen--Macaulay of codimension $s$).
The kernel of the multiplication having the described form is standard, as is the isomorphism $\Sigma^s \Dual(R/K) \wequi \Hom_A(R/K, A)$.

\ref{it:ex_dual_1} The map $R \otimes_A B \to B$ is surjective with kernel generated by $T_i \otimes 1 - 1 \otimes t_i$, where $t_i$ is the image of $T_i$ in $B$; indeed it is a base change of $R \hat\otimes_A R \to R$.
Since $B$ is finite flat over $A$, the dimension of $R \otimes_A B$ is the same as that of $R$, and hence $B$ has codimension $s$ on $R \otimes_A B$, just as on $R$.
It follows that the $T_i \otimes 1 - 1 \otimes t_i$ form a regular sequence on $R \otimes_A B$; denote by $\Kos(T-t)$ the corresponding Koszul complex.
This is a resolution of $B$ by free $(R \otimes_A B)$-algebras, whence projective $R$-algebras.
Denote by $\Kos(a)$ the Koszul complex of the sequence $a_1, \dots, a_s$.
Since $\Kos(T-t)$ is a projective resolution, the composite $\Kos(a) \to R/J \to R/K$ lifts along $R/K \xleftarrow{\wequi} \Kos(T-t)$ to a map $\zeta: \Kos(a) \to \Kos(T-t)$.
By definition the map $\beta$ is obtained as \[ \Hom_A(B, A) \wequi \pi_0 \Sigma^s \Dual(R/K) \wequi \pi_{-s} \Hom_R(\Kos(T-t), R) \xrightarrow{\zeta^*} \pi_{-s} \Hom_R(\Kos(a), R) \wequi R/J. \]
Recall that $\Kos(a)$ is defined as an exterior algebra over $R$; call the generators $e_1, \dots, e_s$.
Similarly denote the exterior algebra generators of $\Kos(T-t)$ over $R \otimes_A B$ by $e_i'$.
The map specified by $\zeta(e_i) = \sum_j b_{ij} e_j'$ is a lift of the desired form.
The claim now follows by direct computation.

\ref{it:ex_dual_2} By Lemma \ref{lemm:explicit}, in the case of a geometric residual intersection the map $I^{t+1}/JI^t \to R/J$ is injective, its image coincides with the image of $\beta$, and $(\tau^\dagger)^{-1}$ is obtained from $\beta$ by restricting the range.
Since $B$ is flat over $A$, $B \otimes_A I^{t+1}/JI^t \to B \otimes_A R/J$ is also injective.
In fact $B$ is free over $A$, say with basis $b_1, \dots, b_n$.
Then we can write $\bar\Delta = \sum_k b_k \otimes i_k$, with $i_k \in R/J$ uniquely determined.
Since $\beta(b_k^\dual) = i_k \in I^{t+1}/JI^t$, the result follows.
\end{proof}

\subsection{Globalization} \label{subsec:globalization}
We now globalize the previous results.
For appropriate schemes $X$ (generalizing $\Spec(R)$) and quasi-coherent ideal sheaves $I \subset J \subset \scr O_X$ with quotient ideal sheaf $K=J:I$, we shall define what it means for this to be a (geometric) residual intersection.
We will define a locally constant function $\ul{t}$ on $X$ corresponding to the number $t=s-g$ from the previous subsections.
(We will also show that even though in the previous formulation it looked like there was a choice, the number $t$ is actually determined by $J$ and $I$.)
Then we construct a coherent sheaf $I^{\ul t+1}/JI^{\ul t}$ on $X$ (if $U \subset X$ is open and $\ul t$ is constant on $U$ with value $t$, then the restriction of $I^{\ul t+1}/JI^{\ul t}$ to $U$ is given by $I^{t+1}/JI^t$) and show that there is a global duality isomorphism for $\scr O_X/K$ involving $I^{\ul t+1}/JI^{\ul t} \otimes \det(J/JK)^\dual$.

\subsubsection{The functions $s, g, t$}
Let $R$ be a Gorenstein noetherian local ring and $J \subset I$ ideals.
Assume that $R \ne K$, where $K := J:I$.
Suppose that $s' \ge s \ge \codim I$ and that $(J,I)$ is both an $s$-residual intersection satisfying the weak hypothesis and an $s'$-residual intersection satisfying the weak hypothesis.
It follows from Proposition \ref{prop:conormal-free} that the module $J/JK$ is free of rank both $s$ and $s'$ over the non-zero commutative ring $R/K$, and hence $s=s'$.
In other words, the ``parameter'' $s$ in an $s$-residual intersection satisfying the weak hypothesis is actually determined by the pair $(J,I)$.\NB{In particular there is \emph{no} choice.}
This justifies the following.
\begin{definition}
Let $R$ be a noetherian local ring and $J \subset I$ ideals.
We put $s(J,I) = \mu(J)$, $g(J,I) = \codim I$, $t(J,I) = \max\{s(J,I)-g(J,I),-1\}$.
We call $K=J:I$ a \emph{residual intersection} if $R \ne K$ and $\codim K \ge s(J,I)$.
We call the residual intersection \emph{geometric} if $\codim (K+I) > s(J,I)$.
We say that $(J,I)$ satisfies condition $G$ if $I$ satisfies $G_{s(J,I)}$, and we say that $(J,I)$ satisfies condition $E_i$ if condition $D_{t(J,I)+i}$ holds for $I$ (this is vacuous if $t(J,I)+i \le 0$).
We say that $(J,I)$ satisfies the standard (respectively weak or strong) condition if it satisfies conditions $G$ and $E_0$ (respectively $E_{-1}, E_1$).
\end{definition}

This definition is slightly more general than what we had before, because we do not require that $s(J,I) > g(J,I)$; equivalently we allow $t(J,I)=-1$.
The following result shows that the additional generality is largely illusory.
\begin{lemma} \label{lem:degenerate-case}
Suppose that $K=J:I$ is a residual intersection with $t(J,I) = -1$.
Then $J$ is generated by a regular sequence and $J=K$.
\end{lemma}
\begin{proof}
We have $\codim I > \mu(J) \ge \codim J$ (the former by assumption and the latter by the principal ideal theorem \cite[Theorem 10.2]{Eisenbud_CommutativeAlgebra}).
Hence $\codim K = \codim J$ (see the end of \S\ref{subsec:quotient-geometric}) and so $\codim J = \mu(J)$, which implies that $J$ is generated by a regular sequence and that $R/J$ is Cohen--Macaulay \cite[Proposition 18.13]{Eisenbud_CommutativeAlgebra} \cite[Theorem 2.1.2(c)]{bruns1998cohen}.
Now $I/J$ has positive codimension on $R/J$, and so contains a regular element \cite[Theorem 18.7]{Eisenbud_CommutativeAlgebra}.
It follows that $K/J = (0) :_{R/J} I/J = 0$, as needed.
\end{proof}
\begin{remark} \label{rmk:generalized-res-int}
It follows that Proposition \ref{prop:conormal-free} (freeness of the canonical module), Lemma \ref{lemm:I-resn} (exactness of the modified Koszul complex) and Theorem \ref{thm:canonical-trace} (existence of canonical traces) remain true in the above situation (i.e. $t<0$), provided that we interpret $I^0/JI^{-1}$ to mean $R/K$ (which is the same as $R/J$).
In other words $J/J^2$ is a free $R/J$-module on $a_1, \dots, a_s$ \cite[Theorem 1.1.8]{bruns1998cohen}, $\Kos(a_1, \dots, a_s) \wequi R/J$ \cite[Corollary 17.5]{Eisenbud_CommutativeAlgebra}, and $\Sigma^s \Dual(R/J) \wequi R/J$ via the self-duality of $\Kos(a_1, \dots, a_s)$ \cite[Proposition 17.15]{Eisenbud_CommutativeAlgebra}.
\end{remark}

\subsubsection{Semicontinuity of $s, g, t$}
We globalize the previous definitions as follows.
\begin{definition}
Let $X$ be a locally noetherian scheme and $J \subset I$ quasi-coherent ideal sheaves.
Define functions $\ul{s} = \ul{s}(J,I), \ul{g} = \ul{g}(J,I), \ul{t} = \ul{t}(J,I)$ on $Z(J:I)$ by $\ul{s}(J,I)(x) = s(J_x, I_x)$ and so on.
We call $K=J:I$ a \emph{(geometric) residual intersection} if for every $x \in Z(K)$, $(J_x,I_x)$ is a (geometric) residual intersection.
We define similarly the conditions $G$, $E_i$ as well as the weak, standard and strong hypotheses by requiring them to hold pointwise on $Z(K)$.
\end{definition}

In good cases, all of the above properties are open on $X$, and hence may be checked on closed points (since $X$ is locally noetherian, any point has a closed specialization \cite[Tag 02IL]{stacks-project}).
We begin with the following preparation.
\begin{lemma} \label{lemm:semicont}
Let $X$ be a locally noetherian scheme and $Z \subset X$ closed.
Consider the following functions
\begin{enumerate}
\item \label{codimI} $x \mapsto \codim_{\scr O_{X,x}} I(Z)_x$
\item \label{muI} $x \mapsto \mu_{\scr O_{X,x}}(I(Z)_x)$, i.e. the minimal number of generators
\item \label{depth-dim} $x \mapsto \depth \scr O_{Z,x} - \dim \scr O_{Z,x}$.
\end{enumerate}
Then \ref{codimI} is lower semicontinuous, \ref{muI} is upper semicontinuous, and \ref{depth-dim} is lower semicontinuous provided that $X$ locally embeds into a regular scheme.
\end{lemma}
\begin{proof}
\ref{codimI} is \cite{EGAIV}, Corollaire 0.14.2.6(ii)], \ref{muI} is \cite[Example III.12.7.2]{hartshorne2013algebraic}, and \ref{depth-dim} is \cite[Proposition 6.11.2(i)]{EGAIV}.
\end{proof}

\begin{corollary} \label{cor:resint-open}
Let $X$ be a locally noetherian scheme and $J \subset I$ quasi-coherent sheaves of ideals.
Consider the subsets of $X$ consisting of those points $x \in X$ satisfying one of the following conditions (for some fixed $i$)
\begin{enumerate}
\item \label{it:resint-open:1} $(J_x,I_x)$ is a (geometric) residual intersection.
\item \label{it:resint-open:2} $I_x$ satisfies condition $G_i$
\item \label{it:resint-open:3} $(J_x,I_x)$ satisfies condition $G$
\item \label{it:resint-open:4} $t(J_x,I_x) < i$
\item \label{it:resint-open:5} $I_x$ satisfies condition $D_i$
\item \label{it:resint-open:6} $(J_x,I_x)$ satisfies condition $E_i$
\end{enumerate}
Then the sets \ref{it:resint-open:1}, \ref{it:resint-open:2}, \ref{it:resint-open:3}, \ref{it:resint-open:4} are open, and \ref{it:resint-open:5}, \ref{it:resint-open:6} are open provided that $X$ locally embeds into a regular scheme.
\end{corollary}
\begin{proof}
\ref{it:resint-open:1} The condition is $\codim K_x - \mu(J_x) \ge 0$, and in addition $$\codim (K+I)_x - \mu(J_x) > 0$$ in the geometric case.
The result holds since these functions are lower semicontinuous by Lemma \ref{lemm:semicont}\ref{codimI}\ref{muI}.

\ref{it:resint-open:2} Let $C_i$ be the locus of points where $I$ requires at least $i$ generators.
Then $C_i$ is closed by Lemma \ref{lemm:semicont}\ref{muI}.
One checks that condition $G_i$ is equivalent to $\codim (C_j \cap Z(I)) \ge j$ for $j \le i$ (see also \cite[p. 312]{ArtinNagata72}, where this is used as the definition of $G_s$), whence the result follows from Lemma \ref{lemm:semicont}\ref{codimI}.

\ref{it:resint-open:3} Condition $G$ is the conjunction of the following conditions for $i \ge 0$: ($\mu(J) < i$ or $G_i$).
Each of these is open by \ref{it:resint-open:2} and Lemma \ref{lemm:semicont}\ref{muI}.
Since $\mu(J)$ is bounded near every point, the intersection is finite, and hence also open.

\ref{it:resint-open:4} The function $t(J_x,I_x)$ is upper semicontinuous by Lemma \ref{lemm:semicont}\ref{codimI}\ref{muI}, whence the result.

\ref{it:resint-open:5} Since $\dim(R/I) = \dim(R/I^j)$ for all $j \ge 1$ \cite[\S8.1 Axiom D2]{Eisenbud_CommutativeAlgebra}, the condition is equivalent to $\depth(R/I^j) - \dim(R/I^j) \ge 1-j$, which is open by Lemma \ref{lemm:semicont}\ref{depth-dim}.

\ref{it:resint-open:6} $E_i$ is the conjunction of the following conditions: ($t(J,I) < j$ or $D_{j+i}$).
Again these conditions are open separately by \ref{it:resint-open:4} and \ref{it:resint-open:5}, and the intersection is locally finite.

\end{proof}

\subsubsection{Local constancy of $s, g, t$}
We have seen (Lemma \ref{lemm:semicont}) that the functions $\ul{s}, -\ul{g}$ and $\ul{t}$ are upper semicontinuous.
In the case of a residual intersection satisfying the strong hypothesis, they are even locally constant.

We begin with a categorified version of local constantness of $\ul{s}$.
\begin{proposition}\label{pr:global_J/KJ_loc_free}
Let $X$ be a Gorenstein scheme\footnote{So by definition $X$ is locally noetherian.} and $J \subset I$ quasi-coherent ideal sheaves.
Put $K = J:I$ and $\schemeW=Z(K)$.
Assume that for every closed point $w \in \schemeW$, $(J_w,I_w)$ is a residual intersection satisfying the weak hypothesis.

Then $J/JK$ is a locally free $\scr O_\schemeW$-module.
In fact, if $a_1, \dots, a_s$ is a minimal generating set for $J$ locally around some (not necessarily closed) point $w \in \schemeW$, then the $a_i$ determine a local basis for $J/JK$.
\end{proposition}
\begin{proof}
Let $w \in \schemeW$ and $a_1, \dots, a_s$ generate $J_w$.
Pick a closed specialization $w'$ of $w$ \cite[Tag 02IL]{stacks-project}.
By assumption, Proposition \ref{prop:conormal-free} and Remark \ref{rmk:generalized-res-int}, $(J/KJ)_{w'}$ is free of rank $s'=\mu(J_{w'}) \ge s$.
Since $\scr O_{\schemeW,w}^s \to (J/KJ)_w \wequi ((J/KJ)_{w'})_w \wequi \scr O_{\schemeW,w}^{s'}$ is surjective it must be an isomorphism (e.g. use \cite[Theorem 2.4]{Matsumura-CommutativeAlg}).
Since $X$ is locally noetherian, the stalkwise generators $a_1, \dots, a_s$ of $J$ and the corresponding stalkwise basis for $J/KJ$ propagate to an open neighborhood of $w$ (see e.g. \cite[Tag 056J]{stacks-project} and \cite[Proposition 1.4]{Eisenbud_CommutativeAlgebra}).
This concludes the proof.
\end{proof}

\begin{lemma} \label{lemm:loc-const}
Let $X$ be a Gorenstein scheme and $K=J:I$ a residual intersection.
\begin{enumerate}
\item\label{it:uls_locally_constant} Suppose that the weak hypothesis holds.
  Then $\ul{s}$ is locally constant (on $W=Z(K)$, where it is defined).
\item\label{it:ult_locally_constant} Suppose that the strong hypothesis holds.
  Then $\ul{t}$ is locally constant on $Z(K+I)$.
\end{enumerate}
\end{lemma}
\begin{proof}
\ref{it:uls_locally_constant} We have $\ul{s}(x) = \Rrank(J_x/K_xJ_x)$. Thus the result follows because $J/JK$ is locally free by Proposition \ref{pr:global_J/KJ_loc_free}.

\ref{it:ult_locally_constant} It suffices to prove that $\ul t$ is constant along generalizations.\NB{E.g.: A sober topological space with finitely many irreducible components is connected if and only if any two points can be connected by a chain of generalizations/specializations. Proof is basically trivial.}
We may thus assume that $X$ is local with closed point $x$.
If $\ul{t}(x) \ge 0$ then $D_1$ holds at $x$, and so $R/I$ is Cohen--Macaulay.
It follows from Lemma \ref{lemm:equicod} below that $\ul{g}(J,I)$ is constant on $Z(I)$, which implies the desired result in this case.
Now assume that $\ul{t}(x) = -1$.
Then since $\ul{t} \ge -1$ and $x$ is the closed point, $\ul{t}$ is constantly equal to $-1$ on $Z(K)$ by upper semicontinuity.
\end{proof}

\begin{lemma} \label{lemm:equicod}
Let $R$ be a local Cohen--Macaulay ring and $I$ an ideal such that $R/I$ is Cohen--Macaulay.
Then every component of $Z(I)$ has the same codimension in $X$.
\end{lemma}
\begin{proof}
Let $I \subset P$ be a minimal prime, corresponding to a component of $Z(I)$.
Then $\codim P + \dim R/P = \dim R$ since $R$ is Cohen--Macaulay \cite[Corollary 2.1.4]{bruns1998cohen}; hence we need only prove that all components of $Z(I)$ have the same dimension.
This holds since $R/I$ is Cohen--Macaulay \cite[Corollary 18.11]{Eisenbud_CommutativeAlgebra}.
\end{proof}

\subsubsection{Globalized duality theorem}
We will now globalize the duality theorem.
Suppose we are in the situation of Lemma \ref{lemm:loc-const}(2), i.e. the strong hypothesis holds and so $\ul{t}$ is locally constant on $Z(K+I)$.
Consider the sequence of sheaves on $X$ \[ F_{-1} = \scr O_X/K \leftarrow F_0 = I/J \leftarrow \dots \leftarrow F_i = I^{i+1}/JI^i \leftarrow \dots. \]
Let $U \subset X$ be open.
Denote $t(U)$ the maximum of $\ul t$ on $U \cap Z(I+K)$ if this set is non-empty, and $t(U) = -1$ else. 
Set \[ F^p_{\ul t}(U) := F_{t(U)}(U), \] where $F_\infty := 0$.
Note that if $V \subset U$ then $t(V) \le t(U)$, so $F^p_{\ul t}$ is a \emph{pre}sheaf in an evident way.
\begin{lemma}
The sheaf $F_{\ul t}$ associated with the presheaf $F^p_{\ul t}$ is coherent.
In fact if $U \subset X$ is open and $U \cap Z(I+K)$ has at most one connected component then $F_{\ul t}|_U \wequi F_{t(U)}|_U$.
\end{lemma}
\begin{proof}
Coherence being a local property \cite[Tag 01XZ(5)]{stacks-project}, it suffices to establish the second claim (such open sets forming a basis of the topology of $X$).
For this it is enough to show that $F^p_{\ul t}|_U \wequi F_{t(U)}|_U$.
Let $V \subset U$.
We must prove that $F_{t(V)}(V) \wequi F_{t(U)}(V)$ (via the canonical map).
Note that all the maps between the $F_i$ are isomorphisms on $X \setminus Z(K)$.
The same is true on $X \setminus Z(I)$ because there $J=K$ by Lemma \ref{lem:degenerate-case}.
Hence the only way our map can fail to be an isomorphism is if $V$ meets $Z(I+K)$.
But then the same holds for $U$, and hence $t(U) = t(V)$ since $\ul t|_{U \cap Z(I+K)}$ is constant by Lemma \ref{lemm:loc-const}.
\end{proof}
We denote the sheaf constructed above by $I^{\ul t+1}/JI^{\ul t}$.

As a final preparation, note that if $F$ is a sheaf on $Z$ and $\ul e$ is a locally constant function, then $\Sigma^{\ul e} F \in \Dcat(Z)$ makes sense.
\begin{theorem}\label{thm:global_canonical_trace}
Let $X$ be a Gorenstein scheme and $K = J:I$ a residual intersection satisfying the strong hypothesis.
Put $\schemeW=Z(K)$.

There exists a unique map\NB{$\Sigma^{-\ul s}\det(J/KJ)^\dual \wequi \Dual_{\schemeW}(\widetilde\det J/KJ)$} \[ \tmap: \Sigma^{-\ul s} I^{\ul t+1}/JI^{\ul t} \otimes_{\calO_{\schemeW}} \det(J/KJ)^\dual \to \calO_X \] such that for every closed point $w \in \schemeW$, after localization to $\scr O_{X,w}$ the map $\tmap$ is given by the map of Theorem \ref{thm:canonical-trace} and Remark \ref{rmk:generalized-res-int}.
Moreover the following hold.
\begin{enumerate}
\item \label{global_canonical_trace:1} The unique map $\tmap$ satisfies this property at all points $w \in \schemeW$.
\item  \label{global_canonical_trace:2} The map  \[ \tmap^\dagger: \Sigma^{-\ul s} I^{\ul t+1}/JI^{\ul t} \otimes_{\calO_{\schemeW}} \det(J/KJ)^\dual \to \Dual(\calO_{\schemeW}) \] adjoint to the composite \begin{gather*} \calO_{\schemeW} \otimes^L \Sigma^{-\ul s}(I^{\ul t+1}/JI^{\ul t} \otimes_{\calO_{\schemeW}} \det(J/KJ)^\dual) \to \calO_{\schemeW} \otimes \Sigma^{-\ul s}(I^{\ul t+1}/JI^{\ul t} \otimes_{\calO_{\schemeW}} \det(J/KJ)^\dual) \\ \xrightarrow{\text{mult}} \Sigma^{-\ul s} I^{\ul t+1}/JI^{\ul t} \otimes_{\calO_{\schemeW}} \det(J/KJ)^\dual \xrightarrow{\tmap} \calO_X \end{gather*} is an isomorphism.
\item  \label{global_canonical_trace:3} Suppose that $\ul s = s$ and $\ul t = t$ are constant on $\schemeW \cap Z(I)$, $V$ is a vector bundle of rank $s$ on $X$, $\sigma$ a section of $V$ with $\image(\sigma^\dual) = J$.
  Then the canonical map \[ V^\dual|_\schemeW \to J/JK \] is an isomorphism, \[ \Kos'(V, \sigma) \to I^{t+1}/JI^t \] is an equivalence in some open neighborhood of $\schemeW$, and after restriction to this open neighborhood the trace $\tmap$ takes the form \[ \tmap_{V,\sigma}: \Sigma^{-s} I^{t+1}/JI^t \otimes_{\calO_{\schemeW}} \det(J/KJ)^\dual \wequi \Sigma^{-s} \Kos'(V, \sigma) \otimes_{\scr O_X} \det{V} \to \scr O_X. \]
\end{enumerate}
\end{theorem}
\begin{remark}
Assuming part  \ref{global_canonical_trace:2} of the theorem, we find that \[ \Map(\Sigma^{-\ul s}I^{\ul t+1}/JI^{\ul t} \otimes_{\calO_{\schemeW}} \det(J/KJ)^\dual, \calO_X) \wequi \Map(\Dual(\calO_{\schemeW}), \calO_X) \wequi \Map(\scr O_X, \scr O_{\schemeW}) \] is discrete.
Thus uniqueness of $t$ has only one possible meaning.
\end{remark}
\begin{proof}
Let $F = \RiHom(\Sigma^{-\ul s}I^{\ul t+1}/JI^{\ul t} \otimes_{\calO_{\schemeW}} \det(J/KJ)^\dual, \calO_X) \in \Dcat(X)$.
We eventually seek to prove that $F$ is the dual of the dual of $\scr O_\schemeW$, i.e. $F \wequi \scr O_\schemeW$.
We shall first show that $\ul{\pi}_i F = 0$ for $i > 0$.
This can be checked on stalks, and formation of $F$ is compatible with passage to stalks \cite[Proposition III.6.8]{hartshorne2013algebraic}. 
If $x \in X \setminus \schemeW$ then $F_x = 0$.
On the other hand if $x \in \schemeW$ then the claim follows from Theorem \ref{thm:canonical-trace} and Remark \ref{rmk:generalized-res-int}.
It follows that $\pi_0 F$ (i.e. the presheaf sending $U \subset X$ to the set of homotopy classes of maps from $\Sigma^{-\ul s}I^{\ul t+1}/JI^{\ul t} \otimes_{\calO_{\schemeW}} \det(J/KJ)^\dual|_U$ to $\scr O_U$) is a sheaf (indeed $\pi_0 F \wequi \Omega^\infty F$ is a $0$-truncated sheaf of spaces).
Because of the existence of closed specializations \cite[Tag 02IL]{stacks-project}, the map \[ \pi_0 F(X) \to \prod_{x \in X_{(0)}} (\pi_0 F)_x \wequi \prod_{x \in \schemeW_{(0)}} (\pi_0 F)_x \] is an injection (here the product is over all closed points).
Noting that the desired map $\tmap$ corresponds to a global section of $\pi_0 F$, this proves uniqueness; i.e. there is at most one map $\tmap$ as in \ref{global_canonical_trace:1}.

Next we prove \ref{global_canonical_trace:3}; in particular we shall show that under the additional assumptions $\tmap$ exists and is given by $\tmap_{V,\sigma}$.
The first isomorphism is immediate from Proposition \ref{prop:conormal-free} and Remark \ref{rmk:generalized-res-int}.
Note that we have $I^{\ul t+1}/JI^{\ul t} \wequi I^{t+1}/JI^t$ on all of $\schemeW$ (not just $\schemeW \cap Z(I)$), since this holds outside $Z(I)$ for trivial reasons.
The map $\Kos'(V, \sigma) \to I^{t+1}/JI^t$ is an equivalence at all closed points of $\schemeW$ by Lemma \ref{lemm:I-resn} and Remark \ref{rmk:generalized-res-int}, thus in some open neighbourhood of $\schemeW$ as claimed.\NB{Using locally noetherian assumption...}
The map $\tmap_{V,\sigma}$ has the required specialization at all points of $\schemeW$, by construction.
Thus by uniqueness, it is the global trace.

From this we deduce existence of $\tmap$ in general.
Since $\pi_0 F$ is a sheaf, we may work locally: it suffices to prove that given $x \in X$ there exists an open neighbourhood $U_x$ of $x$ and a map $\tmap_x$ over $U_x$ such that for every point $x' \in U_x$ the map induced by $\tmap_x$ on the stalk at $x'$ is the one coming from Theorem \ref{thm:canonical-trace} or Remark \ref{rmk:generalized-res-int}.
Thus let $x \in X$.
If $x \not\in \schemeW$ we put $U_x = X \setminus \schemeW$ and $\tmap_x = 0$.
Now suppose that $x \in \schemeW$.
Pick minimal generators $a_1, \dots, a_s$ for $J_x$.
They extend to generators of $J$ over some sufficiently small open neighbourhood $U$ of $x$.
Shrinking $U$, we may assume that $U \cap Z(I+K)$ has at most one connected component, and thus $\ul s$ and $\ul t$ are constant thereon (Lemma \ref{lemm:loc-const}).
We are thus in the situation of \ref{global_canonical_trace:3} (take $V = \scr O^s$, $\sigma = (a_1, \dots, a_s)$), whence the local trace $\tmap_x$ exists.
It follows that $\tmap$ exists on all of $X$, as needed.

We have in fact arranged that our map $\tmap$ has the expected stalk at all points $x \in X$, whence \ref{global_canonical_trace:1}.
It remains to prove \ref{global_canonical_trace:2}, i.e. that $\tmap^\dagger$ is an equivalence.
This can be checked stalkwise, where it holds by construction.
\end{proof}

\begin{remark} \label{rmk:identify-i!}
Write $i: \schemeW \to X$ for the canonical inclusion.
The exceptional functor $$i^!: \Dcat(X) \to \Dcat(\schemeW)$$ satisfies $i_*i^!(\ph) = \RiHom(\scr O_\schemeW, \ph)$ \cite[Tag 0AA2]{stacks-project}.
Since $i_*$ is $t$-exact and conservative, we thus learn from (2) that $\Sigma^{\ul s}i^!(\scr O_X)$ is concentrated in degree $0$, namely given (canonically) by \[ \Sigma^{\ul s}i^!(\scr O_X) \wequi I^{\ul t+1}/JI^{\ul t} \otimes_{\calO_{\schemeW}} \det(J/KJ)^\dual. \]
In the situation of (3) we have $\widetilde\det V^\dual|_\schemeW \wequi \Sigma^s \det(J/KJ) $, and so the above can more compactly be written as \[ i^!(\widetilde \det V^\dual) \wequi I^{\ul t+1}/JI^{\ul t}. \]
\end{remark}

\section{Almost complete intersections} \label{sec:almost-complete-int}
We apply the results of the previous section in the case of \emph{almost complete intersections} and use them to globalize a result of van Straten--Warmt \cite{van2011gorenstein}.
\begin{warning}
In this section, we denote by $\Dual$ various (shifted and twisted) duality functors.
This is in contrast with the previous sections, where $\Dual$ always denoted $\RiHom(\ph, R)$.
\end{warning}
\subsection{Almost complete intersections after van Straten--Warmt}
The following is mainly a summary of \cite[\S\S3,4]{van2011gorenstein}.
Since we generalize some statements slightly, we choose to supply a full proof.
\begin{proposition} \label{prop:vSW}
Let $R$ be a regular local ring of dimension $n$, $f = (f_1,\dots, f_n) \in R^n$ a sequence of elements such that the ideal $J$ generated by $f$ has dimension $\le 1$.
Write $K=\Kos(f)$ for the Koszul complex and $\Dual=\RiHom(\ph,\Sigma^n R)$ for the canonical shifted duality, so that $\Dual K \wequi K$.
\begin{enumerate}
\item \label{it:vSWpr1}We have $\pi_i K = 0$ for $i \not\in \{0,1\}$.\NB{This holds more generally if $R$ is arbitrary and $J$ has depth $n-1$, e.g. $R$ CM and $J$ dimension $1$.}
  Put $A=\pi_0 K$ ($=R/J$) and $B=\Sigma \pi_1 K$, so that we have a cofiber sequence $B \to K \to A$ (in $\Dcat(R)$).
\item \label{it:vSWpr2} We have $\pi_0 \Dual B \wequi R/I$ for some ideal $I \supset J$, $\pi_i \Dual B = 0$ else, $\pi_0 \Dual A \wequi I/J$, $\pi_1 A \wequi \pi_1 K$, $\pi_i A = 0$ else.
  All these equivalences are induced by the natural maps $\Dual A \to \Dual K \wequi K$ and $K \wequi \Dual K \to \Dual B$.
\item  \label{it:vSWpr3}We have $I/J=H^0_m(R/J)$, i.e. $I$ is the saturation of $J$.
\end{enumerate}
\end{proposition}
\begin{proof}
\ref{it:vSWpr1} By \cite[Theorems 17.4 and 17.6]{Eisenbud_CommutativeAlgebra}, we need to prove that $J$ has depth $\ge n-1$.\NB{I.e. contains a regular sequence of length $n-1$.}
Since $R$ is Cohen--Macaulay (so the depth of $J$ coincides with $\codim J$ \cite[Theorem 18.7]{Eisenbud_CommutativeAlgebra}), this follows via \cite[Corollary 18.10]{Eisenbud_CommutativeAlgebra} (implying that $\dim J + \codim J = \dim R$) from our assumption about the dimension of $J$.

\ref{it:vSWpr2} Recall that if $R$ is any ring and $M$ is a finitely generated $R$-module with annihilator of depth $c$, then $\Ext^i(M, R) = 0$ for $i < c$ \cite[Proposition 18.4]{Eisenbud_CommutativeAlgebra}, and if $R$ is regular local of dimension $n$, the same vanishing holds for $i > n$ \cite[Corollary 19.6]{Eisenbud_CommutativeAlgebra}.
Since $\pi_* K$ is supported on $J$, i.e. has annihilator of depth $\ge n-1$, we deduce from this that $\pi_* \Dual A = 0$ for $* \not \in \{0,1\}$ and $\pi_* \Dual B = 0$ for $* \not\in \{0,-1\}$.
The long exact sequence associated with the cofiber sequence $\Dual A \to \Dual K \wequi K \to \Dual B$ shows that $\pi_1 \Dual A \wequi \pi_1 K$ via the canonical map, $\pi_{-1} \Dual B = 0$, and the canonical sequence \[ 0 \to \pi_0 \Dual A \to \pi_0 K \to \pi_0 \Dual B \to 0 \] is exact.
Since $\pi_0 K = R/J$, the existence of $I$ as claimed follows.

\ref{it:vSWpr3} Since $\Dual B \wequi R/I$ we have $\Ext^n(R/I, R) \wequi \pi_0 \Dual \Dual B \wequi \pi_0 B = 0$ and hence $H^0_m(R/I) = 0$ by local duality \cite[Corollary 3.5.9]{bruns1998cohen}.
Considering the long exact local cohomology sequence associated with the short exact sequence above, it will thus suffice to show that $I/J \wequi H^0_m(I/J)$.
This follows from the fact that $I/J \wequi \pi_0 \Dual A = \Ext^n(R/J, R)$.
Indeed for any finitely generated module $M$ over a Gorenstein ring $R$ of dimension $n$ we have $H^0_m(\Ext^n(M,R)) \wequi \Ext^n(M,R)$, as we now prove.
Let $\hat R$ be the completion.
Since both $\Ext(M, \ph)$ \cite[Tags 0A6A(3) and 066E(3)]{stacks-project} and $H^0_m$ \cite[Proposition 3.5.4(d)]{bruns1998cohen} commute with base change along the faithfully flat \cite[Tag 00MC]{stacks-project} map $R \to \hat R$, we may assume that $R$ is complete.
By local duality in the complete case \cite[Theorem 3.5.8]{bruns1998cohen} we get $\Ext^n(M,R) \wequi \Hom_R(H^0_m(M), E)$ (for a certain module $E$), which is annihilated by some power of $m$ (as desired) since $H^0_m(M)$ is (being a submodule of $M$ and so finitely generated).
\end{proof}

\begin{remark}
This result is essentially trivial if $J$ has dimension $0$, i.e. is a complete intersection.
Indeed then $B=0$, $I=R$ and $R/J$ is artinian.
\end{remark}

\subsection{Almost complete intersections: global case}
Recall the notion of associated points \cite[Tag 02OI]{stacks-project}.
\begin{lemma} \label{lem:remove-embedded-points}
Let $X$ be a locally noetherian scheme.
\begin{enumerate}
\item \label{remove-embedded-points:1} There exists a unique quasi-coherent sheaf of ideals $I_0 \subset \scr O_X$ such that for every closed point $x \in X$ we have $I_{0,x} = H^0_x(\scr O_{X,x})$.
\item \label{remove-embedded-points:2} $Z(I_0) \subset X$ is the largest closed subscheme such that none of the associated points of $\scr O_{Z(I_0)}$ is closed.
\end{enumerate}
\end{lemma}
\begin{proof}
\ref{remove-embedded-points:1} Uniqueness is clear by existence of closed specializations \cite[Tag 02IL]{stacks-project}.
Let $E$ be the set of finite subsets of closed points of $X$, and set $I = \colim_{V \in E} \ul\Gamma_V(\scr O)$.
Since formation of $H^0_V$ commutes with localization \cite[Tag 07ZP(2)]{stacks-project} and $V \cap \{x\} = \{x\}$ for any closed point $x$ and all $V \in E$ sufficiently large, $I$ satisfies the desired property.

\ref{remove-embedded-points:2} Again by existence of closed specializations, we may assume that $X$ is local.
Thus it suffices to prove that if $R$ is a noetherian local ring, then $I_0 = H^0_m(R) \subset R$ is the smallest ideal such that $m$ is not an associated prime of $R/I_0$.
That $m$ is not associated on $R/I_0$ follows from \cite[Proposition 3.13]{Eisenbud_CommutativeAlgebra}.
Conversely, let $I$ be an ideal and $x \in I_0 \setminus I$.
We must show that $m$ is associated on $R/I$.
By assumption $x\cdot R/I$ is non-zero and if $P$ is one of its associated primes then (since $x \in I_0$) $m^n \subset P$ for some $n$, so $P=m$.
This concludes the proof.
\end{proof}

\begin{remark} \label{rmk:remove-embedded}
We illustrate the effect of the construction $Z(I_0) \hookrightarrow X$.
\begin{enumerate}
\item\label{rmk:remove-embedded-i} If $x \in X$ is closed, then $x \in Z(I_0)$ if and only if $x$ is not isolated.
  Equivalently, if $R$ is a noetherian local ring, then $H^0_m(R) \subset m$ if and only if $\dim R > 0$ \cite[Tag 00KH]{stacks-project}.
\item\label{rmk:remove-embedded-ii} Let $x \in Z(I_0)$ be closed.
  Then $Z(I_0) \to X$ is an isomorphism near $x$ if and only if $x$ is not embedded.
  Equivalently, $H^0_m(R) = 0$ if and only if $m$ is not associated on $R$ \cite[Proposition 3.13(a)]{Eisenbud_CommutativeAlgebra}.
\item Let $\schemeW = X:Z(I_0)$ be the scheme theoretic difference.
  Then $\schemeW$ consists of closed points (i.e. $\dim \schemeW = 0$); equivalently $(0) : H^0_m(R)$ has dimension $0$ (since it contains some power of $m$).
  By \ref{rmk:remove-embedded-i} and \ref{rmk:remove-embedded-ii}, these closed points must be either isolated or embedded (equivalently associated and closed).
  This set is locally finite \cite[Tag 05AF]{stacks-project}.
\end{enumerate}
\end{remark}

\begin{example} \label{ex:I-CM}
If $\dim X = 1$, then $Z(I_0)$ is Cohen--Macaulay.
Indeed a $1$-dimensional scheme is Cohen--Macaulay if and only if it has no embedded points \cite[Tag 0BXG]{stacks-project}.
\end{example}

Now let $X$ be a Gorenstein scheme, equidimensional of dimension $s$, and $Z \subset X$ closed of dimension $1$.
Put $J = I(Z)$ and denote by $Z_0 \subset Z$ the closed subscheme constructed in Lemma \ref{lem:remove-embedded-points}.
Let $I = I(Z_0)$, $K = J:I$ and $\schemeW = Z:Z_0$ ($=Z(K)$).
Then:
\begin{itemize}
\item $\schemeW = Z:Z_0$ is an ($s$-)residual intersection if and only if $J$ is locally at closed points of $\schemeW$ generated by $s$ elements (indeed $\schemeW$ has codimension $s$ by Remark \ref{rmk:remove-embedded}(3)).
\item $I$ satisfies the strong hypothesis (equivalently the standard hypothesis) if and only if it is generically a complete intersection (being Cohen--Macaulay by Example \ref{ex:I-CM}), if and only if $J$ is a complete intersection at its generic points of dimension $1$.
  This holds for example if the components of $J$ of dimension $1$ have multiplicity $1$.
\end{itemize}

\begin{theorem}\label{thm:EU-vSW}
Let $X$ be a regular scheme, equidimensional of dimension $s$, $V$ a vector bundle on $X$ of rank $s$ such that $Z = Z(\sigma)$ has dimension $1$.
Let $I,J,K,\schemeW,Z_0$ be as above and assume that $\schemeW=Z:Z_0$ is an $s$-residual intersection satisfying the strong hypothesis.
Finally assume that $2$ is invertible on $X$.
Write $i: \schemeW \to X$ for the closed immersion.

Then we have \[ [e(V, \sigma)] = i_*(c) \in \W_Z^V(X), \] for a certain class $c \in \W(\schemeW, i^! \widetilde \det V^\dual)$.
In fact $i^! \widetilde \det V^\dual \wequi I^2/JI$ (by Remark \ref{rmk:identify-i!}) and $c$ is the class corresponding to the multiplication map $I/J \otimes I/J \to I^2/JI$.
\end{theorem}
See Section~\ref{subsection:notation_conventions} for information on the Euler class $[e(V, \sigma)] $ in Witt theory. We give more information on Witt theory as we use it now. There is a Witt group $\W(X, \scr L)$ twisted by any shifted line bundle (or even more generally dualizing complex) $\scr L$; one has \[ \W^V(X) := \W(X, \widetilde \det V^\dual) := \W^{rk(V)}(X, \det V^\dual). \]
Given any closed immersion $i: \schemeW \to X$, there is a pushforward \[ i_*: \W(\schemeW, i^! \widetilde \det V^\dual) \to \W_\schemeW(X, \widetilde \det V^\dual) \to \W_Z(X, \widetilde \det V^\dual). \]

\begin{remark}
Note that $\schemeW$ is a disjoint union of ``fat points'' (being locally noetherian of dimension $0$).
We thus obtain a formula for $[e(V,\sigma)]$ as a sum over local terms, one for each point of $Z$ that is either isolated or embedded.
In fact near isolated points $i$ is a regular immersion, $I^2/JI = \scr O_\schemeW$ and $c=1$; the novelty of our result is the terms at embedded points.
\end{remark}

\begin{example}\label{ex:O3+O2onP2_sectionsyx3}
The equality $i_* c = [e(V,\sigma)] $ does not necessarily hold in $\GW_{Z}^V(X).$ For example, let $X = \P^2_k = \Proj k[x,y,z]$, with structure map $\pi: X \to \Spec k$. Let $V$ be the vector bundle $V = \calO(2) \oplus \calO(3)$ and let $\sigma$ be the section given by $\sigma = (xy, x^3)$. Then $J$ is the sheaf of ideals $J = (xy, x^3)$. We let $Z_0 \subset Z$  the closed subscheme constructed in Lemma \ref{lem:remove-embedded-points} and $I = I(Z_0)$ as before. Then $I = (x)$. (To see this, note that away from $[0:0:1]$ the equality holds by Remark~\ref{rmk:remove-embedded}~\eqref{rmk:remove-embedded-ii} as $(x,y)$ is the only embedded prime. At $[0:0:1],$ we compute $H^0_{(x,y)}\big(\frac{k[x,y]_{(x,y)}}{(xy,x^3)} \big)\cong \cup_{m} \{ a \in \frac{k[x,y]_{(x,y)}}{(xy,x^3)}  : a f= 0 \text{ for all }f \in (x,y)^m\} \cong  (x)$.) The isomorphisms \begin{equation}\label{relorex:O3+O2P2xyx3}\widetilde \det V^\dual \cong \calO(5)[2] \cong \omega_{X/k}[2] \otimes \calO(4)^{\otimes 2}\end{equation} define a pushforward $\pi: \GW_{Z}^V(X) \cong  \GW_{Z}^{T_{X/k}}(X) \to \GW(k) $, and the Euler number of $V$, with respect to the relative orientation given by \eqref{relorex:O3+O2P2xyx3}, is $n(V) := \pi_* [\Kos(V,\sigma)]$.  Since $V$ contains an odd dimensional summand, $n(V)$ is a multiple of the hyperbolic form (see for example \cite[Proposition 19]{FourLines}). Since its classical Euler number is $6$, we have $n(V) = 3(\lra{1} + \lra{-1})$. On the other hand, $I/J = kx \oplus kx^2$ is rank $2$, whence so is $\pi_*[I/J]$, preventing a possible equality $i_* c = [i_*(I/J)] = [\Kos(V,\sigma)] $ in $\GW_{Z}^V(X).$ Moreover, because the non-degenerate form on $I/J$ passes through the multiplication $I/J \times I/J \to I^2/IJ = kx^2 \oplus kx^3$, it follows in an elementary manner that $(\pi_Z)_*c$ is of the form \begin{equation*} (\pi_Z)_*c = \begin{pmatrix} 
* & * \\
* & 0 
\end{pmatrix}\end{equation*} where $\pi_Z : Z \to \Spec k$ is the structure map. It then follows that $(\pi_Z)_*c = \lra{1} + \lra{-1}$ is the hyperbolic form in $\GW(k)$ as in \cite[Lemma C.1]{BW-A1Eulerclasses}.
\end{example}

\begin{remark}
The proof of Theorem \ref{thm:EU-vSW} shows that assuming the strong hypothesis is in some sense non-essential (and so is assuming $1/2 \in X$).
Namely, even without the strong hypothesis, we have $[e(V,\sigma)] = [I/J]$, for \emph{some} Poincar\'e structure (i.e., bilinear form) on the module $I/J$.
(Indeed the proof constructs an explicit algebraic surgery on a bilinear form representing the Euler class, and observes that the underlying module of the resulting form is $I/J$.)
The strong hypothesis (and $1/2 \in X$) allows us to identify the Poincar\'e structure in terms of multiplication and trace maps.
\end{remark}

\subsection{Review of Poincar\'e $\infty$-categories}\label{subsection:Poincare_infinity_categories}
A \emph{Poincar\'e $\infty$-category} is a pair $(\scr C, \Qoppa)$ of a stable $\infty$-category $\scr C$ and a functor $\Qoppa: \scr C^\op \to \SH$ satisfying certain assumptions \cite[Definitions 1.2.1 and 1.2.8]{CDHHLMNNS-1}.
One puts \[ B_\Qoppa: \scr C^\op \times \scr C^\op \to \SH, (X, Y) \mapsto \Qoppa(X \oplus Y)/(\Qoppa(X \oplus 0) \oplus \Qoppa(0 \oplus Y); \] the axioms imply that this is \emph{bilinear} and in fact of the form \[ B_\Qoppa(X, Y) = \map(X, D_\Qoppa(Y)) \] for a uniquely determined equivalence \[ D_\Qoppa: \scr C^\op \to \scr C. \]
The space $\Omega^\infty \Qoppa(X)$ is called the \emph{space of forms on $X$}.
Via the map $\Qoppa(X) \to B_{\Qoppa}(X,X)$, any form $q$ induces a map $X \to D_\Qoppa X$; the form is called \emph{non-degenerate} (or $(X,q)$ is a \emph{Poincar\'e object}) if this map is an equivalence.

\begin{example}
Let $X$ be a scheme, $\scr L$ a line bundle and $n \in \Z$.
We put $\scr C = \Dperf(X)$ (the $\infty$-category of perfect complexes \cite[1.2.12]{CDHHLMNNS-1}) and \[ \Qoppa(E) = \map(E \otimes E, \Sigma^n\scr L)^{hC_2}. \]
One may show that this satisfies the axioms and one obtains $D_\Qoppa(E) = \iHom(E, \Sigma^n \scr L)$.
\end{example}

\begin{remark}
The above Poincar\'e structure is not in general the one we want, for the following reasons.
\begin{enumerate}
\item If $2$ is not invertible on $X$, this $\Qoppa$ is not correct.
  Instead one should use the ``genuine symmetric'' Poincar\'e structure which can be obtained from \cite[Proposition 4.2.15, Notation 4.2.20]{CDHHLMNNS-1} by gluing.
\item We do not really want to work with $\Dperf(X)$ but rather the bounded coherent derived category $\Dbc(X)$.
  Indeed this is where Grothendieck duality using dualizing complexes happens.
\end{enumerate}
\end{remark}

\begin{example}\label{ex:Qoppa-Dbc}
Suppose that $X$ is a locally noetherian scheme on which $2$ is invertible.
Then given a dualizing complex $\scr D$ on $X$, the category $\Dbc(X)$ acquires a Poincare structure with $\Qoppa_{\scr D}(E) = \map(E \otimes E, \scr D)^{hC_2}$.
For any closed subscheme $Z \subset X$, this restricts to a Poincar\'e structure on the subcategory $\Dbc_Z(X)$ of complexes supported on $Z$.
If $p: X \to Y$ is a proper morphism and $\scr D$ is a dualizing complex on $Y$, then $p^! \scr D$ is a dualizing complex on $X$ and \[ \pi_*: (D^{b,c}(X), \Qoppa_{p^! \scr D}) \to (D^{b,c}(Y), \Qoppa_{\scr D}) \] upgrades to a Poincar\'e functor.
\end{example}

An important construction in any Poincar\'e $\infty$-category is \emph{algebraic surgery}.
Namely given a Poincar\'e object $(X, q)$ (i.e. $X \in \scr C$ and $q \in \Omega^\infty \Qoppa(X)$) together with $f: T \to X$ and a null homotopy $\eta$ of $f^*q \in \Omega^\infty \Qoppa(T)$ one may construct a (``cobordant'') Poincar\'e object $(X_f, q_f)$.
Essentially by definition, $(X,q)$ and $(X_f, q_f)$ have the same image in the Witt group of $(\scr C, \Qoppa)$.
We can describe $(X_f, q_f)$ as follows.
Let $F$ be the fiber of $X \wequi D_\Qoppa X \to D_\Qoppa T$.
The image of the null homotopy $\eta$ in $B_\Qoppa(T,T) \wequi \map(T, D_\Qoppa T)$ provides a lift of $f: T \to X$ to $T \to F$; then $X_f$ is the cofiber of this map.
In particular we have maps $X_f \leftarrow F \to X$ yielding \[ \Omega^\infty \Qoppa X_f \to \Omega^\infty \Qoppa F \leftarrow \Omega^\infty \Qoppa X. \]
While $q_f$ is a bit tricky to describe in general, it at least has the property that its image in $\Omega^\infty \Qoppa F$ is canonically homotopic to the image of $q$.

\subsection{Proof of Theorem \ref{thm:EU-vSW}}
Recall that $e(V,\sigma) \in \GW_Z(X, \widetilde \det V^\dual)$ corresponds to the Koszul complex $K=\Kos(V, \sigma)$, with its usual shifted duality.
We denote by $$\Dual = \RiHom(\ph, \widetilde \det V^\dual)$$ the duality functor.
We work in the Poincar\'e $\infty$-category $(\Dbc_Z(X), \Qoppa_{\widetilde \det V^\dual})$ of Example \ref{ex:Qoppa-Dbc}.

Set $B = \tau_{\ge 1} K$ and $A = \tau_{\le 0} K \wequi \scr O/J$.
We claim that
\begin{itemize}
\item $\ul\pi_i K = 0$ for $i \not\in \{0,1\}$,
\item $K \wequi \Dual K \to \Dual B$ identifies $\Dual B \wequi \scr O/I$, and
\item $\Dual A \to \Dual K \wequi K$ identifies $\ul \pi_0 \Dual A = I/J$, $\ul \pi_1 \Dual A = \ul \pi_1 K$, $\ul\pi_i \Dual A = 0$ else.
\end{itemize}
By existence of closed specializations and compatibility of $\RiHom$ with passage to stalks \cite[Proposition III.6.8]{hartshorne2013algebraic}, for this we may assume that $X$ is local.
(Note in particular that the second point says that $\ul\pi_i \Dual B = 0$ for $i \ne 0$ and the map $\scr O/J = \ul\pi_0 K \to \ul\pi_0 \Dual B$ is a surjection with kernel $I/J$, which \emph{can} be checked locally. Similarly for the third point.)
The claim thus reduces to Proposition \ref{prop:vSW}.

The map $B \to K \wequi \Dual K \to \Dual B \wequi \scr O/I$ is canonically null, source and target being in orthogonal parts of the $t$-structure.
We may thus perform algebraic surgery on this map.
In other words $B \to K$ lifts canonically into the fiber $\Dual A$ of $K \to \Dual B$, and if $C$ denotes the cofiber of $B \to \Dual A$, then $C$ carries a canonical non-degenerate form cobordant to the one on $K$.
By direct computation we have $C \wequi I/J$, which is the underlying object of $i_*(c)$.
It remains to show that the form produced on $C$ by algebraic surgery is the canonical one.

Since $2$ is invertible, a homotopy class of forms on $E$ is a homotopy class of maps $E \otimes E \to \widetilde \det V^\dual$ or equivalently $E \to \Dual E$ (which are symmetric up to homotopy).
The second description shows that homotopy classes of forms on $I/J$ are certain (homotopy classes of) maps $I/J \to I/J$, whereas homotopy classes of forms on $\Dual A$ are certain homotopy classes of maps $\Dual A \to \Dual \Dual A \wequi A \wequi \scr O/J$, i.e. (for $t$-structure reasons) maps $I/J = \ul\pi_0 \Dual A \to \scr O/J$.
We deduce that the canonical map \[ \Hom(I/J, I/J) \wequi \pi_0 \Omega^\infty \Qoppa(I/J) \to \pi_0 \Omega^\infty \Qoppa(\Dual A) \wequi \Hom(I/J, \scr O/J) \] is an injection.
In other words the form on $C \wequi I/J$ is determined uniquely by the condition that its pullback to $\Dual A$ coincides with the pullback of the form on $K$, up to homotopy.
We shall show that the form on $I/J \wequi i_*(c)$ satisfies the same property.
By the first description of homotopy classes of forms, it will thus be enough to establish commutativity of the following square
\begin{equation*}
\begin{CD}
\Dual A \otimes \Dual A @>a>> K \otimes K \\
@VbVV                 @Vq_KVV     \\
I/J \otimes I/J @>T>> \widetilde \det V^\dual,
\end{CD}
\end{equation*}
where $T$ is the form on $i_*(c)$, $q_K$ is the form on $K$, and $a, b$ are the tensor squares of the canonical maps $\Dual A \to I/J$ (projection to $\pi_0$) and $\Dual A \to K$ (dual to $K \to \tau_{\le 0}K = A$).
Let $L \subset K$ be the subcomplex where in the lowest term we have put $I$ instead of $\scr O$ (so this is the complex $\Kos'(V,\sigma)$ of Definition \ref{def:K'} for $t=0$); then $K/L \wequi \scr O/I$ and so $L \wequi \Dual A$.
Let $K' \subset L \subset K$ be the complex with $I^2$ in degree $0$ and $I \otimes \Lambda^2 V$ in degree $1$; i.e. this is the complex $K'=\Kos'(V, \sigma)$ of Definition \ref{def:K'} for $t=1$.
We have $K' \wequi I^2/JI$ by Theorem \ref{thm:global_canonical_trace}(3).
Consider the following commutative diagram
\begin{equation*}
\begin{CD}
I/J \otimes I/J @<b'<< L \otimes L @>a'>> K \otimes K \\
@VmVV                 @VmVV              @VmVV    \\
I^2/J @<{\wequi}<< K' @>>> K @>>> \widetilde \det V^\dual.
\end{CD}
\end{equation*}
Here the maps denoted $m$ are given by multiplication, $a'$ is the tensor square of the canonical inclusion $L \to K$, and $b'$ is the tensor square of the projection $L \to I/J \wequi \pi_0 L$.
The composite from top left to bottom right via inverting the equivalence is the trace form $T$, by Theorem \ref{thm:global_canonical_trace}(3).
The composite from top right to bottom right is the form $q_K$.
Using that under $L \wequi \Dual A$ the maps $(a,b)$ correspond to $(a',b')$, this establishes the desired commutativity.
This concludes the proof.

\appendix
\section{Macaulay2 code}\label{Section:M2code}
We have implemented functions for the computer program Macaulay2 \cite{M2} to carry out the computations implicit in Proposition \ref{A1Eulernumber_excess_bundle_n_quadrics_containingP2} for concrete examples, in the case $n=5$.
Specifically, we have a function \texttt{Bprime} which can be used to obtain the matrix of a symmetric bilinear form $B'$ such that $B' \wequi \BH \oplus 5\lra{0}$, and we have a function \texttt{diagonalize} which can be used to diagonalize a symmetric bilinear form over a field of characteristic $\ne 2$. The function \texttt{Bprime} uses a supporting function \texttt{CDTr} which is based on a function of S. Pauli \cite{pauli2020computing}. The code is included at the end of this appendix.

\begin{example}
Below is a sample session computing a diagonal representative of $B'$ for a random set of five quadrics on $\P^5$ vanishing on $\P^2$, over the finite field $\FF_{61}$.
\begin{verbatim}
Macaulay2, version 1.16
with packages: ConwayPolynomials, Elimination, IntegralClosure, InverseSystems,
LLLBases, MinimalPrimes, PrimaryDecomposition, ReesAlgebra, TangentCone, Truncations

i1 : load "form.m2"

i2 : load "diagonalization.m2"

i3 : FF = GF 61;

i4 : R = FF[x,y,z];

i5 : M = random(R^5, R^{3:-1})

o5 = | 20x+4y-21z  5x-30y-24z  18x+11y-20z |
     | -16x-22y+5z -2x+20y-25z -7x+25y+23z |
     | 28x-7y+26z  30x+21y+7z  27x-24y-30z |
     | 5x-12y-2z   x-26y+24z   14x+18y     |
     | 21x+2y+z    2x-14y-4z   8x+26y+27z  |

             5       3
o5 : Matrix R  <--- R

i6 : Bil = Bprime(FF,R,M);

              9        9
o6 : Matrix FF  <--- FF

i7 : (base,D)=diagonalize(Bil);

i8 : D

o8 = | -1 0   0   0   0 0 0 0 0 |
     | 0  -11 0   0   0 0 0 0 0 |
     | 0  0   -29 0   0 0 0 0 0 |
     | 0  0   0   -12 0 0 0 0 0 |
     | 0  0   0   0   0 0 0 0 0 |
     | 0  0   0   0   0 0 0 0 0 |
     | 0  0   0   0   0 0 0 0 0 |
     | 0  0   0   0   0 0 0 0 0 |
     | 0  0   0   0   0 0 0 0 0 |

              9        9
o8 : Matrix FF  <--- FF
\end{verbatim}

We thus find that \[ [\BH] = \lra{-1} + \lra{-11} + \lra{-29} + \lra{-12} \in \W(\FF_{61}). \]
\end{example}

\begin{lstlisting}[lastline=29,breaklines=true,postbreak=\mbox{\textcolor{red}{$\hookrightarrow$}\space},caption=Functions \texttt{Bprime} and \texttt{CDTr} from \texttt{form.m2}]
CDTr=(FF,C,I,f) -> (
  r:=degree I;
  B := basis(C/I);
  B2:=mutableMatrix B;
  J:=determinant jacobian I;
  toVector := q -> last coefficients(q, Monomials=>B);
  E:=J_(C/I)/r;
  p:=0;
  j:=0;
  while j<r do (if (toVector E)_(j,0)!=0 then p=j;j=j+1);
  B2_(0,p)=E;
  B2=matrix(B2);
  T:=mutableIdentity(C/I,r);
  i:=0;
  while i<r do (T_(i,p)=(toVector E)_(i,0) ; i=i+1);
  T=matrix T;
  T1:=T^(-1);
  linear := v -> v_(p,0);
  lift(linear(T1*(toVector f_(C/I))),FF)
)

Bprime = (FF, R, M) -> (
  D := for i from 0 to 2 list det submatrix'(M,{0,3,4},{i});
  E := for i from 0 to 2 list det submatrix'(M,{0,1,4},{i});
  f := {submatrix'(M,{3,4},{}), submatrix'(M,{0,4},{}), submatrix'(M,{0,1},{})}/det;
  I := ideal(f);

  matrix flatten for k from 0 to 2 list for l from 0 to 2 list flatten for i from 0 to 2 list for j from 0 to 2 list CDTr(FF,R,I,(-1)^(i+k)*D_i*E_k* (generators R)_j * (generators R)_l)
)
\end{lstlisting}
\begin{lstlisting}[lastline=35,firstline=4,breaklines=true,postbreak=\mbox{\textcolor{red}{$\hookrightarrow$}\space},caption=Function \texttt{diagonalize} from \texttt{diagonalization.m2}]
diagonalize = mat -> (
  M := mat.target;
  assert (M === mat.source);
  n := M.numgens;
  g := new MutableList from (for k from 1 to n list (generators M)_{k-1});
  bilform := (x, y) -> (
    ((transpose y) * mat * x)_(0,0)
  );
  for k from 0 to n-1 do (
    -- this loop modifies the generators by adding some multiple of one to another
    -- in particular, the modified generators form a basis at every step
    curgen := g#k;
    if bilform(curgen,curgen) == 0 then (
      for l from k+1 to n-1 do (
        nextgen := g#l;
        a := bilform(nextgen,nextgen);
        b := bilform(curgen,nextgen);
        if a == 0 then continue;
        if 2*a == b then (curgen = curgen + 2*nextgen) else (curgen = curgen + nextgen);
        break;
      );
    );
    g#k = curgen;
    a := bilform(curgen,curgen);
    if a == 0 then continue; -- if we are still self-orthogonal, must be orthogonal to everything
    for l from k+1 to n-1 do (
      g#l = (g#l - curgen*matrix{{bilform(curgen,g#l)/a}});
    );
  );
  -- the form should be diagonal with respect to the new basis
  (for k from 0 to n-1 list g#k,matrix for k from 0 to n-1 list for l from 0 to n-1 list bilform(g#k,g#l))
)
\end{lstlisting}

\bibliographystyle{alpha}
\bibliography{bibliographyexcess}

\end{document}